\documentclass[10pt]{article}
\title{On minimal representation-infinite algebras}
\author{Klaus Bongartz\thanks{klausbongartz@online.de}\\ Universit\"{a}t Wuppertal}
\date{\begin{center}{\itshape Irrtum verl\"{a}sst uns nie, doch ziehet ein h\"{o}her Bed\"{u}rfnis immer den strebenden Geist leise zur Wahrheit hinan}\end{center} 
\hspace*{6cm} {\itshape Xenien}} 

\usepackage{amsthm} 
\usepackage[utf8]{inputenc}

\newtheorem{theorem}{Theorem}
\newtheorem{lemma}{Lemma}

\newtheorem{proposition}{Proposition}

\begin{document}
\maketitle
\begin{abstract}
 We consider finite dimensional basic associative algebras over an algebraically closed field and we classify those   that are not distributive and minimal 
representation-infinite.
 As a consequence the number of
isomorphism  classes of all  minimal representation-infinite
 algebras of any fixed dimension is finite and  there are ${\bf Z}$-forms for these. 
We  show that tame concealed algebras are minimal representation-infinite and that the classification of all minimal representation-infinite algebras would lead to a useless unreadable list.

\end{abstract}

\section{Introduction}

Our algebras  $A$ are basic, associative and of finite dimension over an algebraically closed field $k$. Such  an $A$ is given by its quiver $Q$
and an admissible ideal $I$. The $A$-modules are  left-modules of finite dimension and we think of these often as representations of $Q$ satisfying the relations imposed by $I$. 
The category of these 
modules is denoted by $mod\,A$.
An algebra $A$ is  representation-finite if it has only finitely many isomorphism classes of indecomposable modules  and minimal representation-infinite if it is not representation-finite, 
but any proper quotient  is. Finally $A$ is distributive if its lattice of two-sided ideals is distributive.

In 1957 Jans showed in \cite{Jans} that a non-distributive algebra is strongly unbounded i.e. that there exist infinitely many $d$ such that there are infinitely many isomorphism classes of indecomposables of dimension $d$.
 Furthermore he mentions two 
conjectures of Brauer and Thrall: The first says that $A$ is representation-finite if there is a bound on the dimensions of indecomposables and the second  says that otherwise 
$A$ is  strongly unbounded. 

The first conjecture was 1968 positively answered by Roiter in \cite{Roiter} using brilliant elementary arguments and for the  generalization to artinian rings 1974 in \cite{Auslander} Auslander
  considered  almost split sequences in disguise.
 The proof \cite{Bautista} of the second conjecture by Bautista in 1985 required
some of the new concepts of 
representation theory introduced after 1968 and also an intensive study of representation-finite and distributive minimal
 representation-infinite algebras. This was done between 1970 and 1985  by several people who turned their attention afterwards to other directions.

However, some natural questions remained unanswered e.g.: Can there be gaps in the lengths
 of the indecomposables? Is there a finite dimensional representation-infinite algebra which is smallest with respect to representation embeddings? Are there 
only finitely many isomorphism 
classes of minimal representation-infinite algebras in each dimension?

 I answered the first two questions in  two former publications \cite{Bind,Bemb} and here 
I  answer the third. To this end  we define five families of algebras depending on parameters by a picture of 
their quivers and by 
giving afterwards the relations and the possible values of the parameters.

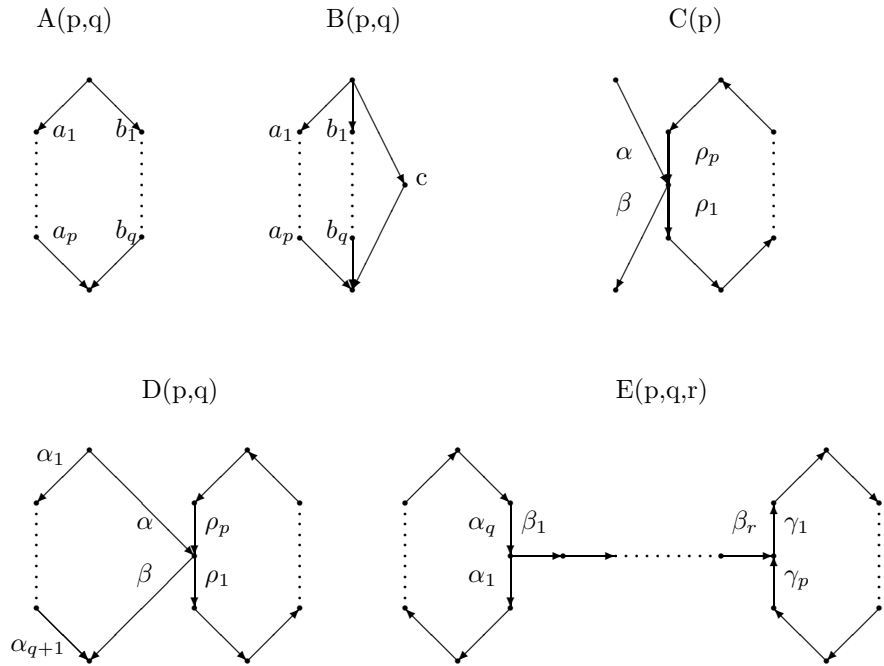
\setlength{\unitlength}{0.7cm}
\begin{picture}(15,7)

\put(0,5){A(p,q)}
\put(1,4){\circle*{0.1}}
\put(0,3){\circle*{0.1}}
\put(2,3){\circle*{0.1}}
\put(2,1){\circle*{0.1}}
\put(0,1){\circle*{0.1}}
\put(1,0){\circle*{0.1}}
\put(1,4){\vector(-1,-1){1}}
\put(1,4){\vector(1,-1){1}}
\put(0,1){\vector(1,-1){1}}
\put(2,1){\vector(-1,-1){1}}
\multiput(0,3)(0,-0.2){10}{\circle*{0.01}}
\multiput(2,3)(0,-0.2){10}{\circle*{0.01}}

\put(0.3,2.9){$a_{1}$}
\put(1.5,2.9){$b_{1}$}
\put(0.3,1){$a_{p}$}
\put(1.5,1){$b_{q}$}

\put(5.5,5){B(p,q)}
\put(6,4){\circle*{0.1}}
\put(5,3){\circle*{0.1}}
\put(7,2){\circle*{0.1}}
\put(6,3){\circle*{0.1}}
\put(6,1){\circle*{0.1}}
\put(5,1){\circle*{0.1}}
\put(6,0){\circle*{0.1}}
\put(6,4){\vector(-1,-1){1}}
\put(6,4){\vector(1,-2){1}}
\put(6,4){\vector(0,-1){1}}
\put(5,1){\vector(1,-1){1}}
\put(6,1){\vector(0,-1){1}}
\put(7,2){\vector(-1,-2){1}}
\multiput(5,3)(0,-0.2){10}{\circle*{0.01}}
\multiput(6,3)(0,-0.2){10}{\circle*{0.01}}
\put(7.2,2){c}
\put(4.4,2.9){$a_{1}$}
\put(5.5,2.9){$b_{1}$}
\put(4.4,1){$a_{p}$}
\put(5.5,1){$b_{q}$}

\put(12,5){C(p)}
\put(14,3){\circle*{0.1}}
\put(14,1){\circle*{0.1}}
\multiput(14,3)(0,-0.2){10}{\circle*{0.01}}
\put(12.5,2.5){$\rho_{p}$}
\put(12.5,1.5){$\rho_{1}$}
\put(11,2.5){$\alpha$}
\put(11,1.5){$\beta$}
\put(12,1){\circle*{0.1}}
\put(12,3){\circle*{0.1}}
\put(13,4){\circle*{0.1}}
\put(13,0){\circle*{0.1}}
\put(12,2){\circle*{0.1}}

\put(11,4){\circle*{0.1}}
\put(11,0){\circle*{0.1}}

\put(14,3){\vector(-1,1){1}}\put(11,4){\vector(1,-2){1}}\put(12,2){\vector(-1,-2){1}}\put(13,4){\vector(-1,-1){1}}\put(12,3){\vector(0,-1){1}}\put(12,2){\vector(0,-1){1}}\put(12,1){\vector(1,-1){1}}

\put(13,0){\vector(1,1){1}}
\end{picture}

\begin{picture}(10,7)

\put(2,5){D(p,q)}
\put(5,3){\circle*{0.1}}
\put(5,1){\circle*{0.1}}
\multiput(5,3)(0,-0.2){10}{\circle*{0.01}}

\put(3,1){\circle*{0.1}}
\put(3,3){\circle*{0.1}}
\put(4,4){\circle*{0.1}}
\put(4,0){\circle*{0.1}}
\put(3,2){\circle*{0.1}}

\put(1,4){\circle*{0.1}}
\put(1,0){\circle*{0.1}}
\put(0,3){\circle*{0.1}}
\put(0,1){\circle*{0.1}}
\multiput(0,3)(0,-0.2){10}{\circle*{0.01}}

\put(3.2,2.5){$\rho_{p}$}
\put(3.2,1.5){$\rho_{1}$}
\put(1.9,2.5){$\alpha$}
\put(1.9,1.5){$\beta$}
\put(0,3.8){$\alpha_{1}$}
\put(-0.5,0.2){$\alpha_{q+1}$}

\put(5,3){\vector(-1,1){1}}
\put(1,4){\vector(1,-1){2}}
\put(3,2){\vector(-1,-1){2}}
\put(1,4){\vector(-1,-1){1}}
\put(0,1){\vector(1,-1){1}}
\put(4,4){\vector(-1,-1){1}}
\put(3,3){\vector(0,-1){1}}
\put(3,2){\vector(0,-1){1}}
\put(3,1){\vector(1,-1){1}}
\put(0,1){\vector(1,-1){1}}
\put(4,0){\vector(1,1){1}}

\put(11,5){E(p,q,r)}
\put(9,2){\circle*{0.1}}
\put(10,2){\circle*{0.1}}
\put(14,2){\circle*{0.1}}
\put(13,2){\circle*{0.1}}
\put(8,0){\circle*{0.1}}
\put(8,4){\circle*{0.1}}
\put(7,1){\circle*{0.1}}
\put(7,3){\circle*{0.1}}
\put(9,1){\circle*{0.1}}
$\put(9,3){\circle*{0.1}}$
\put(14,1){\circle*{0.1}}
\put(14,3){\circle*{0.1}}
\put(15,0){\circle*{0.1}}
\put(15,4){\circle*{0.1}}
\put(16,1){\circle*{0.1}}
\put(16,3){\circle*{0.1}}
\put(7,3){\vector(1,1){1}}

\put(8,4){\vector(1,-1){1}}
\put(9,3){\vector(0,-1){1}}
\put(9,2){\vector(0,-1){1}}
\put(9,1){\vector(-1,-1){1}}
\put(8,0){\vector(-1,1){1}}

\put(14.2,2.5){$\gamma_{1}$}
\put(14.2,1.5){$\gamma_{p}$}
\put(8.2,1.5){$\alpha_{1}$}
\put(8.2,2.5){$\alpha_{q}$}
\put(9.2,2.5){$\beta_{1}$}
\put(13.2,2.5){$\beta_{r}$}

\put(9,2){\vector(1,0){1}}
\put(10,2){\vector(1,0){1}}
\put(13,2){\vector(1,0){1}}

\put(14,3){\vector(1,1){1}}
\put(15,4){\vector(1,-1){1}}
\put(16,1){\vector(-1,-1){1}}
\multiput(11,2)(0.2,0){10}{\circle*{0.01}}
\put(14,2){\vector(0,1){1}}
\put(14,1){\vector(0,1){1}}
\put(15,0){\vector(-1,1){1}}
\multiput(7,1)(0,0.2){10}{\circle*{0.01}}
\multiput(16,1)(0,0.2){10}{\circle*{0.01}}

\end{picture}

\vspace{1cm}

For the family $A(p,q)$ there is no relation and one has $p,q \geq 0$. In the family $B(p,q)$ only the possibilities $p\geq q=1$ and 
$4 \geq p \geq q=2$ are allowed and the sum of all three paths between the source and the sink is a relation. Thus so far we have just the tame canonical algebras. In the remaining cases all parameters $p,q,r \geq 1$ are allowed.
 $C(p)$  has one zero- relation $\rho_{1}\rho_{p}$ which also holds 
for $D(p,q)$ where in addition the two paths between the source and the sink give a commutativity relation.  Finally 
the relations $\alpha_{1}\alpha_{q}$,$\gamma_{1}\gamma_{p}$ and $\gamma_{1}\beta_{r}\ldots \beta_{1}\alpha_{q}$ define $E(p,q,r )$.

Recall that for each algebra $A$ with a source $a$ and a sink $z$ one obtains another 'glued' algebra by identifying $a$ and $z$
 to one point $x$ and by adding in the new quiver  all paths of length 2 with $x$ as an interior point to the relations.

Our first result says: 

\begin{theorem} 
 An algebra  over an algebraically closed field is basic minimal representation-infinite and not distributive if and only if it is isomorphic to an algebra listed above or to its glued version.
\end{theorem}

This has an interesting consequence whose analogue for representation-finite algebras is not true because non-standard algebras exist.

\begin{theorem} Let $d$ be a natural number. There is a finite list of ${\bf Z}$-algebras which are free of rank $d$ as  ${\bf Z}$-modules such that for each 
algebraically closed field $k$ the algebras
$A\otimes_{{\bf Z}} k$ form a list of basic minimal representation-infinite algebras of dimension $d$.
\end{theorem}

In fact both theorems and their  proofs remain valid for $k$-split algebras over any field.
 
 The proof of theorem 1 given on 15 pages is the heart of the article. There are similarities to the proof of the structure theorems 
for non-deep contours in the central  article \cite{BGRS} of Bautista, Gabriel, Roiter and Salmer\'{o}n 
 about multiplicative bases. In section 2 we subdivide the problem into three different types of algebras, where the first two are related by the glueing procedure 
described above.  
The algebras of type 2 are analyzed in section three and the slightly more delicate algebras of type 3  in section 4. 

In section 5   theorem 2 is derived from theorem 1 and my former results on coverings in \cite{Bind}. 

In  the next section it is shown that tame concealed algebras are  minimal representation-infinite.  Somewhat surprisingly this is nowhere mentioned in the recent literature. The mathematical and historical relations between the characterization of tame concealed algebras  by Happel and Vossieck and my results on critical simply connected algebras
are clarified.

At the end we prove   that all basic distributive minimal representation-infinite algebras
 can be obtained by a glueing process from  a critical line or a critical simply connected algebra. However, in the second case  a complete classification remains out of reach.

\section{The trichotomy }
\subsection{Notations, conventions and a reminder on distributive algebras}

Throughout this article $A$ is a basic associative algebra of finite dimension over a field $k$,  $N$ denotes the radical of $A$ and $S$ the socle of $A$ as a bimodule.
We assume that $A/N$ is a product of copies of $k$ which holds always if $A$ is basic and $k$ is algebraically closed.
By a fundamental  observation of Gabriel  there is then a
 quiver $Q$ and 
a surjective algebra homomorphism $\pi$ from the path-algebra $kQ$ to $A$
 whose kernel is contained in the ideal generated  by all paths in $Q$ of length $2$. We fix such a presentation and we write often $\overline{v}$ instead of $\pi(v)$. Thus we get in $A$ a 
decomposition of $1$ as a sum of pairwise orthogonal primitive idempotents 
 $1=\ \sum_{x \in Q_{0}} e_{x}$ where $e_{x}$ is the image of the path of length $0$ through the point $x$. 

We denote by $\mathcal{I}$ the lattice of two-sided ideals of $A$, by $\mathcal{I'}$ the sublattice of the ideals contained in $N$ and by $\mathcal{B}(x,y)$ 
the lattice of $e_{x}Ae_{x}-e_{y}Ae_{y}$-subbimodules of $e_{x}Ae_{y}$. For such a subbimodule $J$ we denote by $rad\,J$ the 
radical as a bimodule and the higher radicals $rad\,^{i}J$ are defined by induction. We have $rad \,e_{x}Ae_{x}=e_{x}Ne_{x}$ for any $x$ and $rad\,J= e_{x}(NJ+JN)e_{y}$ for any subbimodule.
The algebra is called distributive provided $\mathcal{I}$ is a distributive lattice.

We refine a little bit 
the important observations of Jans \cite{Jans} and Kupisch \cite{Kupisch1} on distributive algebras.

\begin{proposition}
 
Keeping the above assumptions and notations we have:
\begin{enumerate}
\item If $J$ is a subbimodule of $e_{x}Ae_{y}$ and $\langle J \rangle $ the two-sided ideal generated by $J$ then we have $\langle J \rangle = NJ+JN+J$ and 
$e_{x}\langle J \rangle e_{y} = \langle J \rangle \cap e_{x}Ae_{y} = J$.
 \item The map $I \mapsto e_{x}Ie_{y}$ is a surjective lattice homomorphism from  $\mathcal{I}$  to  $\mathcal{B}(x,y)$ for all $x,y$.

\item For two points $x,y$  the following are equivalent: 
\begin{enumerate}
\item  $\mathcal{B}(x,y)$ is distributive. 
\item $dim \,(rad\,^{i}e_{x}Ae_{y}/rad\,^{i+1}e_{x}Ae_{y}) \, \leq 1$ for all $i$.
\item $e_{x}Ae_{y}$ is a uniserial bimodule i.e. it has a unique chain of subbimodules.
 
\end{enumerate}

\item Equivalent are:
\begin{enumerate}
 \item $\mathcal{I}$ is  distributive.
\item $\mathcal{I'}$ is distributive.
\item All the lattices $\mathcal{B}(x,y)$  are distributive.
\end{enumerate}

\item The ring $e_{x}Ae_{x}$ is uniserial if and only if its radical is $0$ or generated by one element $\alpha_{x}$. 
\item Let $x,y$ be two points such that $e_{x}Ae_{x}$ and $e_{y}Ae_{y}$ are both uniserial.
Then $e_{x}Ae_{y}$ is uniserial as a bimodule if and only if for $i=0$ and $i=1$ we have $dim \,(rad\,^{i}e_{x}Ae_{y}/rad\,^{i+1}e_{x}Ae_{y}) \, \leq 1$. In that case $e_{x}Ae_{y}$ is uniserial as a left $e_{x}Ae_{x}$- or a right $e_{y}Ae_{y}$-module.
\end{enumerate}
\end{proposition}
 \begin{proof}
  Statement i) is immediately clear and also that the map $I \mapsto e_{x}Ie_{y}$ preserves intersections, sums and inclusions. The surjectivity follows from the last equation in i).

Suppose now that  the vector space $V=(rad\,^{i}e_{x}Ae_{y}/rad\,^{i+1}e_{x}Ae_{y}) $ has dimension $\geq 2$ for some $i$. Then $rad\,^{i}e_{x}Ae_{y}$ lies in $N$. Namely for $x\neq y$ we 
have $e_{x}Ae_{y} \subseteq N$ and for $x=y$ we have $i\geq 1$. In $V$ there is a plane containing three different lines violating the law of distributivity. Their preimages 
$L_{1},L_{2},L_{3}$ under the canonical projection are subbimodules also violating distributivity and so $\mathcal{B}(x,y)$ is not distributive. Similarily one gets that $\mathcal{I'}$ is not 
distributive by looking at the two-sided ideals generated by the $L_{i}$ and using part i).

We have just seen that the distributivity of $\mathcal{B}(x,y)$ implies for all $i$ that $ dim (rad\,^{i}e_{x}Ae_{y}/rad\,^{i+1}e_{x}Ae_{y}) \leq 1$. 
It follows easily that $\mathcal{B}(x,y)$ is uniserial whence distributive. So part iii) is true.

If $\mathcal{I}$ is distributive so is its sublattice $\mathcal{I'}$. From this we obtain by the argument from above that $ dim (rad\,^{i}e_{x}Ae_{y}/rad\,^{i+1}e_{x}Ae_{y}) \leq 1$ for 
all $i$ and all $x,y$. Thus all $\mathcal{B}(x,y)$ are distributive by part iii). Using the relation $I= \oplus _{x,y}e_{x}Ie_{y}$ valid for any two-sided ideal one gets that $\mathcal{I}$ 
is distributive.
 
Part v) is trivial.

If one of the spaces $e_{x}Ae_{x}$, $e_{y}Ae_{y}$ or $e_{x}Ae_{y}$ has dimension $\leq 1$ then part vi) is obvious.
In the other case let $a$ be a generator of $e_{x}Ae_{y}$. Then $\alpha_{x}a$ and $a\alpha_{y}$ are not linearly independent modulo $rad\,^{2}e_{x}Ae_{y}$. Up to symmetry
 we can assume that we have 
$\alpha_{x}a= \xi a\alpha_{y} + r$ for some scalar $\xi$ and some $r \in rad\,^{2}e_{x}Ae_{y}$. Then we obtain $\alpha_{x}^{p}a \alpha_{y}^{q}=\xi^{p}a\alpha_{y}^{p+q} + r(p,q)$ with some
 $r(p,q) \in  rad\,^{p+q+1}e_{x}Ae_{y}$ for all $p$ and $q$ by induction on $p$. Now the elements $\alpha_{x}^{i}a\alpha_{y}^{n-i}$  with $ i \geq n$ generate $rad\,^{n}e_{x}Ae{y}$ for 
any $n$ and this space is zero for large $n$.
By descending induction it follows that all $rad\,^{i}e_{x}Ae_{y}$ are generated by the $a\alpha_{y}^{j}$ with $j\geq i$. Thus $e_{x}Ae_{y}$ is cyclic as a module over $e_{y}Ae_{y}$ 
whence uniserial.
 \end{proof}

\subsection{The subdivision }

We show that the minimal non-distributive algebras fall into three  disjoint classes.
 A pair $(a,z)$ of points is called critical if the bimodule $e_{z}Ae_{a}$ is not uniserial. 
The critical index $i(a,z)$ of such a pair pair is then the smallest natural number such that $rad\,^{i}e_{z}Ae_{a}/rad\,^{i+1}e_{z}Ae_{a}$ has dimension $\geq 2$.
Furthermore, given a point $x$ in $Q$, we denote by $I_{x}$ the two-sided ideal generated by all paths of lengths $2$ with $x$ as the interior point.
 Recall that $x$ is called a node if $I_{x} \subseteq I$ holds.
\begin{proposition} Let $A=kQ/I$ be an algebra which is not distributive but any proper quotient is.
Then the following holds:
\begin{enumerate}
 \item For any critical pair $(a,z)$ with critical index $i$ we have $rad\,^{i+1}e_{z}Ae_{a}=0$ and $S(a,z):= rad\,^{i}e_{z}Ae_{a}$ is a  bimodule of dimension 2 which is contained in $S$.
\item There is only one critical pair $(a,z)$ and we have $S=S(a,z)$. 
Moreover we are in one of the following three situations:
\begin{enumerate}
 \item ( type 1 ) $a=z$, $i(a,z)=1$, $e_{a}Ae_{a}\simeq k[X,Y]/(X,Y)^{2}$ and $I_{a}\subseteq  I$.
\item (type 2) $a\neq z$, $i(a,z)=0$, $a$ is a source, $z$ a sink in $Q$ and for $e=e_{a}+ e_{z}$ the algebra $eAe$ is isomorphic to the path-algebra of the Kronecker
 quiver $K_{2}$ consisting of two parallel arrows.
\item (type 3) $a \neq z$, $i(a,z)=1$ and for $e= e_{a} + e_{z}$ the algebra $eAe$ is isomorphic to the path algebra of the quiver with one loop $\alpha$ in $a$, one arrow $\beta$ from $a$ to $z$ 
and one loop $\gamma$ in $z$ divided by the relations $\alpha^{2}= \gamma^{2}= \gamma \beta \alpha =0$.
\end{enumerate}

\end{enumerate}
 
\end{proposition}

\begin{proof} We consider the two-sided  ideal $J$ generated by $rad\,^{i+1}e_{z}Ae_{a}$. Then we have $e_{z}Ae_{a}\cap J =rad\,^{i+1}e_{z}Ae_{a}$ whence the quotient $A/J$ is 
still
not distributive. By minimality  we have $J=0$ and a fortiori $rad\,^{i+1}e_{z}Ae_{a}=0$. Similarly, if $V:=(Nrad\,^{i}e_{z}Ae_{a} \, + \,rad\,^{i}e_{z}Ae_{a}N) \neq 0$ we look at the
non-zero two-sided ideal $J$ it
 generates.
Because of $J\cap e_{z}Ae_{a} = 0$ the proper quotient $A/J$ is again not distributive and so $J=0$ and a fortiori $V=0$. This means that   $rad\,^{i}e_{z}Ae_{a}$ 
is contained in $S$. If the dimension  of 
 $rad\,^{i}e_{z}Ae_{a}$ is strictly greater than $2$  we choose a non-zero subbimodule $J$ of codimension $2$ in $rad^{i}e_{z}Ae_{a}$. Then  $J$ is even a two-sided ideal and  $A/J$ 
is still not distributive. This contradiction shows that $dim \,S(a,z) =2$. 

There is at least one critical pair $(a,z)$ and we have $S=S(a,z)\oplus S'$ for some two-sided ideal $S'$. This ideal is zero because   $A/S'$ is still rerepresentation-infinite. Thus 
we have $S=S(a,z)$ and there is only one critical pair. We discuss the different possibilities.

 For $a=z$ we have $i=i(a,z)=1$ and  $e_{a}Ae_{a} \simeq k[X,Y]/(X,Y)^{2}$.
For any path $p=\beta\alpha$ of length $2$ with interior point $a$
we consider the two-sided ideal $J$ generated by $\overline{p}$. For any paths $v,w$ we have that $e_{a}\overline{v\beta}$ and $\overline{\alpha}wte_{a}$ are in $rad\,e_{a}Ae_{a}$ 
whence their product vanishes and $J\cap e_{a}Ae_{a}=0$. Thus $A/J$ is still not distributive and we have $J=0$ by minimality. Thus we have $I_{a}\subseteq I$.

For $a\neq z$ all $e_{x}Ae_{x}$ are uniserial rings and we can apply the last part of proposition 1 to see that only $i=i(a,z)=0$ and $i=1$ are possible.
In the case $i=0$ we have $S(a,z)=e_{z}Ae_{a}$. Take an element $f$ in some $e_{a}Ne_{y}$. Then the two-sided ideal $J$ generated by $f$ is spanned by products $vfw$ and 
the 
intersection with $e_{z}Ae_{a}$ by products $e_{z}vfwe_{a}$. This product vanishes because $f$ annihilates the element $e_{z}v$ from $S(a,z)$. Thus $A/J$ is still not distributive and we
 conclude $J=0$ whence $f=0$. It follows  that $x$ is a source. Dually  $z$ is a sink and so $eAe$ has the wanted form.

Finally we look at the case $a\neq z$, $i=1$ and $S=rad e_{z}Ae_{x}$. Let $f$ be in $rad\,^{2}e_{a}Ae_{a}$ and let $J$ be the two-sided ideal generated by $f$. Then the intersection of $J$ with 
$e_{z}Ae_{a}$ is spanned by products $e_{z}vfwe_{a}$ which are all $0$. We get that $J=0$ and $f=0$, i.e. $ dim\, e_{a}Ae_{a}=2$  and dually $ dim\, e_{z}Ae_{z}=2$. 
 Let $f$ be an element in $e_{a}Ae_{z}$ such that the intersection of the ideal $J$ generated by $f$ with $e_{z}Ae_{a}$ 
is not $0$. Then there is a product $e_{z}vfwe_{a}\neq 0$. The non-zero products $fwe_{a}$ and $e_{z}vf$ show that the quiver of $eAe$ has no loops and so it is an oriented cycle.
 But then $eAe$ 
is uniserial. Thus the intersection $J\cap e_{z}Ae_{a}$ is zero, $J=0$ and $f=0$. It follows that $eAe$ has the wanted shape.

\end{proof}

A non-distributive algebra satisfies the second Brauer-thrall conjecture as Jans has shown by direct calculations already in \cite{Jans}. 
A  proof of his results via representation embeddings is given in \cite[section 3.1]{Bemb}.

\subsection{Glueing and separating: the relation between the first two types}

We recall and refine a little bit the well-known constructions of glueing a source and a sink or separating a node into a source and a sink  \cite{Martinez}.

So let $Q$ be a quiver with a proper source $a$ and a proper sink $z$. Denote by $Q'$ the quotient obtained by identifying $a$ and $z$ to one point $x$. The other points of $Q'$ and the arrows
are just dashed versions of those of $Q$. It follows from the universality of path-algebras that there is an algebra-homomorphism $\phi:kQ' \longrightarrow kQ$ with $\phi(\alpha'
)= \alpha$ for each arrow, $\phi(e_{y'})=e_{y}$ for all $y$ different from $a$ and $z$ and $\phi(e_{x})=e_{a}+e_{z}$. The image of $\phi$ is the subalgebra $B$ of $kQ$ generated by 
$f=e_{a}+e_{z}$, by the other idempotents and by all arrows. We denote by $I(x)$ the ideal of $kQ'$ generated by all path $\beta'\alpha'$ where $x$ is the end-point of $\alpha'$.
We have $\phi(\beta'\alpha')= \phi(\beta' e_{x}\alpha')= \beta (e_{a}+e_{z})\alpha=0$ because there is no arrow ending in $a$ or starting in $z$. Thus $I(x)$ lies in the kernel of $\phi$.
On the other hand the paths in $Q'$ of length at least $1$ and not having $x$ as an interior point are in bijection under $\phi$ with all proper paths in $Q$. Thus $\phi$ induces an isomorphism
$kQ'/I(x) \simeq B$.

Reversely one can start with a quiver $Q'$ containing a point of transition $x$ and separate this point into an emitter $a$ and a receiver $z$ to obtain a quiver $Q$ with a proper source $a$ 
and a proper sink $z$. Clearly these operations on quivers are inverse to each other.

There is an  exact functor $F$ from $mod\,kQ $ to $mod\, kQ'/I(x)$ defined in the language of representations by $FM(x)=M(a)\oplus M(z)$, $FM(y')=M(y)$ for $y'\neq x$  
and by the obvious action on the arrows. This functor maps indecomposables to indecomposables and it hits all indecomposables up to isomorphism. The two simples corresponding 
to the points $a$ and $z$ are the only two non-isomorphic indecomposables that become isomorphic. In fact $F$ is just the restriction to $B$ if one identifies $B$ with $kQ'/I(x)$.

\begin{proposition}
 We keep all the assumptions and notations introduced above. Let $J$ be a two-sided admissible ideal in $kQ$ such that $A=kQ/J$ is finite-dimensional and let $J'$ be the inverse image of $J$ 
under $\phi$ and define $A'=kQ'/J'$. Then we have:
\begin{enumerate}

 \item $A$ is distributive iff $A'$ is distributive.
\item $A$ is minimal representation-infinite iff $A'$ is so.
\item $A$ is a non-distributive minimal representation-infinite algebra of type 2 with respect to $a$ and $z$ iff $A'$ is one of type 1 with respect to $x$.
\end{enumerate}

\end{proposition}

\begin{proof} For an algebra $C$ we denote by $\mathcal{I'(C)}$ the lattice of two-sided ideals of $C$ contained in the radical. Recall that $C$ is distributive iff $\mathcal{I'(C)}$  
is distributive. Now $\mathcal{I'(A')}$ and $\mathcal{I'(B/J)}$ are isomorphic. Any two-sided ideal of $B/J$ contained in the radical is automatically a two-sided ideal in $A$. Thus part i)
 is proven.

The functor $F$ applied to the full subcategories of representations annihilated by $J$ resp. by $J'$ shows that $A$ is representation-infinite iff $A'$ is so. 
Now a representation-infinite algebra $C$ is minimal representation-infinite iff all quotients $C/I$ with $I$ contained in the radical are representation-finite.
 For $A$ and for $A'$ these ideals correspond each other and the representation types of the quotients coincide. Part ii) follows.

We know already that $A$ is non-distributive minimal representation-infinite iff $A'$ is so.   We have $e_{x}A'e_{x}\simeq f(B/J)f= f (rad\,A)\,f \oplus kf$ and
 $fAf=f (rad\,A)\,f \oplus ke_{a} \oplus ke_{z}$. Part iii) follows from proposition 2 by comparing the dimensions.
 
\end{proof}

\subsection{General remarks on the proof}
The  basic minimal representation-infinite algebras that are not distributive of types 2 or 3  will be studied in the next two sections.
 We call such an algebra suspicious.
The only critical pair is denoted by $(a,z)$ and the two-dimensional two-sided socle by $S$.

We consider the algebra often as a $k$-category with the points of  $Q$ as objects and with the $A(x,y)=e_{y}Ae_{x}$ as morphism spaces.
Any non-zero morphism $f \in A(x,y)$ can be prolongated to a non-zero morphism $gfh \in S$ and so we have $A(a,y)\neq 0 \neq A(y,z)$ for all $y$. A path $p$ in $Q$ 
is called a zero-path resp. a non-zero-path if $\overline{p}=0$ resp. $\overline{p}\neq 0$. Recall that we work with a fixed presentation. Any non-zero path $p$ from $x$ to $y$ can be
 prolongated to a non-zero-path $p_{2}pp_{1}$ 
with
$\overline{p_{2}pp_{1}} \in S$. Such a path is called long.

Observe  that all $A(x,x)$ are uniserial and all $A(x,y)$ are uniserial for $(x,y)\neq (a,z)$. A point $x$ is called thin if $dim \,A(x,x)=1$ and thick otherwise.   Given two morphisms $f,g \in A(x,y)$ we write $f\sim g$ 
if both elements generate the same subspace of $A(x,y)$. For three thin points $x_{1},x_{2},x_{3}$ with $(x_{1},x_{3})\neq (a,z)$ and
 morphisms $f,g \in A(x_{1},x_{2})$, $h \in A(x_{2},x_{3})$ one has a nice cancellation property: $hf \sim hg  \not\sim 0$ implies $f \sim g$. We often use that the situation is self-dual. 
In particular there
is a dual cancellation result.

Our main method to derive all the wanted results is to look at a full subcategory $A'$ of $A$ supported  by $5$ points at most and at its quiver $Q'$.
Then any proper quotient of $A'$ has to be representation-finite. To exclude certain possibilities we will  construct a quotient of $A'$ which is defined by zero-relations. Then there is
 a Galois-covering $\tilde{A'}$ 
given by  an infinite tree with relations. The group is free and it acts freely so that it is sufficient to find a representation-infinite tree-algebra as a full convex subcategory of 
$\tilde{A'}$ and this is always easy.

Our method is based on the elementary part of Galois-coverings as defined by Gabriel ( see \cite[theorem 16, part a)]{BSURVEY} ). It was applied again and again in Gabriels proof 
for the structure and disjointness theorems of non-deep contours ( see \cite[remark 3.8 ]{BGRS} ). But there the situation is more complicated  and one cannot 
always reduce to a 
quotient $A'$ given by zero-relations. In fact 
  the only representation-infinite algebras
we need to know are quiver algebras of types $\tilde{A}_{n}$, $\tilde{D}_{n}$ and $\tilde{E}_{6}$.

\section{Algebras of type 2}
\subsection{Thick points}
Throughout this section $A$ is a suspicious algebra of type 2 with quiver $Q$.
Thus we have a source $a$ and a sink $z$. Let $b$ be a thick point which is of course different from $a$ and $z$.
 We choose a generator $r$ of $rad\, A(b,b)$ as well as generators $s$ resp. $t$ of $A(a,b)$ resp. $A(b,z)$ as modules over $A(b,b)$.
\begin{lemma} 
  The full subcategory $A'$ supported by $a,b,z$ is given by the quiver $Q'$ with  arrows $\alpha:a \rightarrow b$, $\beta:b \rightarrow z$ and 
 $\rho:b \rightarrow b$ and  the relation $\rho^{2}=0$.
\end{lemma}
\begin{proof}   Because $a$ is a source and $z$ is a sink the quiver $Q'$ contains the three arrows mentioned above. We denote by  $n$ be the greatest integer with $r^{n}\neq 0$. Then
 we have $tr^{n}s \neq 0$ by 
the prolongation property. If the elements $tr^{i}s$ with 
$0 \leq i \leq n$ do not generate $S=A(a,z)$ there is an arrow from $a$ to $z$ which implies the contradiction that the separated quiver is a quiver of type $\tilde{A}_{3}$. 
Thus $Q'$ has only three arrows. 

The full subcategory $A''$ supported by $b$ and $z$ is representation-finite with  the quiver containing $\rho$ and $\beta$ and defined by the relation $\rho^{n+1}=0$. 
The universal cover of $A''$ shows that $n\leq 2$ holds.

Suppose $\rho^{2}\neq 0$. Because of $dim\,S=2$ there is  a non-trivial linear relation 
$x_{0} \beta \alpha \, + \,x_{1} \beta \rho \alpha \,+ \,x_{2} \beta \rho^{2} \alpha $.
For $x_{0}\neq 0$ we can replace the presentation $\pi$ by a new presentation $\pi':kQ' \longrightarrow A'$ by defining
 $\pi'(\alpha) = \pi (x_{0}  \alpha \, + \,x_{1} \rho \alpha \,+ \,x_{2}  \rho^{2} \alpha ) $ and so $A'$ is defined by the relations $\rho^{3}$ and  $\beta \alpha$. 
Then the two other paths
 produce a basis of $S$.  Similarly for  $x_{1}\neq 0=x_{0}$ we can  reduce to the relations $\rho^{3}$ and 
$\beta \rho \alpha$. So in all cases $A'$ is defined by zero-relations.  One finds in the corresponding Galois-covering $\tilde{A'}$  as convex subcategories 
for $\overline{\beta\alpha}\neq 0$ a quiver of 
type 
$\tilde{D}_{4}$ or for $\overline{\beta\alpha} = 0$ a tame concealed algebra of type $\tilde{E}_{6}$ which both are annihilated  by all liftings of the  path  $\beta\rho^{2}\alpha$.
 This is a contradiction.
\end{proof}

\begin{lemma}
 We keep the above assumptions and notations.
\begin{enumerate}
 \item There are arrows $\alpha:a \rightarrow b$, $\beta:b \rightarrow z$ in $Q$ and an oriented cycle $\rho:=\rho_{m}\rho_{m-1} \ldots \rho_{1}$ in $b$ of 
length $m\geq 1$ such that $\overline{\beta \alpha}$ and 
$\overline{\beta \rho \alpha}$  give a basis of $S$.
\item $b$ is the only thick point.
\item One has $\overline{\xi\rho_{m}}=0$ for all arrows $\xi \neq \beta$ and therefore $dim \,A(b,x)\leq 1$ for all $x$ with $b\neq x \neq z$.
 
\end{enumerate}

\end{lemma}

\begin{proof} 

We choose a path $\beta$ with $\overline{\beta}=t$. 
If $\beta$ is not an arrow  we choose a decomposition $\beta= \beta_{1}\beta_{2}$ where $\beta_{1}:y \rightarrow z$ is an arrow. Then $a,b,y,z$ are four
 different points in $Q$  ( $ y\neq b$ follows from  $r^{2}=0$ ) and we look at the full subcategory $A'$ supported by these four points and its quiver $Q'$ which contains the 
arrows $\alpha:a \rightarrow b$, $\gamma:b \rightarrow y$ and $\beta_{1}:y \rightarrow z$. The two elements  $\overline{\gamma \alpha}$ and
 $\overline{\gamma \rho \alpha}$ are  linearly independent in $A(a,y)$ because their products with $\overline{\beta_{1}}$ are so in $A(a,z)$. Thus $A(a,y)$ is cyclic over $A(y,y)$ and so we get
$\dim \,A(a,y)=2= dim\,A(y,y)$ from the last lemma and similarly $\dim \,A(a,b)=2= dim\,A(b,b)=dim \,A(b,y)$. It follows that $A(b,y)$ is uniserial from both sides.  Thus there are only two possibilities 
for the quiver $Q'$ of $A'$: Either one adds an arrow 
$\epsilon:y \rightarrow b$  and the relations $(\gamma \epsilon )^{2}= (\epsilon \gamma)^{2}$ hold or one adds a loop $\rho$ in $b$ and a loop $\sigma$ in $y$ and the relation 
$\gamma \rho \,= \sigma \gamma $ holds. In the second case one can even divide by $\sigma \gamma$ and one gets also a zero-relation algebra. In $\tilde{A'}$ one finds 
in both cases easily
 a convex subcategory with quiver an extended Dynkin diagram of type $\tilde{D}_{5}$ that is annihilated by the liftings of 
the path $\beta\rho\alpha$.  This contradiction shows that $\beta$ is an arrow in $Q$ and so is  $\alpha$ by duality.

Let $b'$ be another thick point. Then the quiver $Q'$ of the full subcategory $A'$ supported at $a,b,b',z$ contains the arrows $\alpha:a
\rightarrow  b$, $\beta: b \rightarrow z $, $\alpha':a
\rightarrow  b'$ and  $\beta': b' \rightarrow z $. If $\rho$ or $\rho'$ factorize it contains also arrows $b \rightarrow b'$ and $b' \rightarrow b$ and then the separated quiver 
to the proper quotient $A'/ rad^{2}\,A'$ contains a
 quiver of type $\tilde{A}_{5}$. The same holds if the loops survive.

If part iii) is not true we find a long path $p_{2}\xi\rho_{m}p_{1}$. Because $\alpha$ is an arrow  we have  $\overline{\rho_{m}p_{1}} \sim \overline{\rho\alpha}$ and because $\beta $ is 
an arrow we get $\overline{p_{2}\xi} \sim \overline{\beta\rho}$ i.e. 
$\overline{p_{2}\xi\rho_{m}p_{1}}\sim \overline{\beta \rho^{2}\alpha} \sim 0$. This is a contradiction. The second statement follows because $A(b,x)$ is cyclic over $A(b,b)$.

\end{proof}
\subsection{Uniqueness and disjointness for long paths}
\begin{lemma}
 Let $b$ be the thick point with an oriented cycle $\rho:=\rho_{m}\rho_{m-1} \ldots \rho_{1}$ in $b$ of 
length $m\geq 1$ such that $\overline{\rho}=r$. Then the following holds:
\begin{enumerate}
\item Any long path $p$ starting with $\alpha$  coincides with $\beta \alpha$ or $\beta \rho \alpha$.
\item Any long path $q$  not starting with $\alpha$  has no interior point in common with $\beta \rho \alpha$. 
\end{enumerate}

\end{lemma}

\begin{proof} 

Let $p=\zeta_{n}\zeta_{n-1}\ldots
\zeta_{1}\alpha$ be a third long  path. We will derive a contradiction.
First assume $m=1$. For $\zeta:=\zeta_{1}=\rho$ we would get $\zeta_{2}=\beta$ from part iii) of lemma 2 and so $p= \beta\rho\alpha$. Thus  $\zeta:b \rightarrow d$ is different from $\rho$.  We consider the full subcategory $A'$ supported by the four points $a,b,d,z$ and its quiver $Q'$ 
in which the arrows $\alpha,\beta,\rho,\zeta$ still exist. An arrow $\zeta':d \rightarrow b$ would occur in a long path $p_{2}\zeta'p_{1}$. But $a$ and $d$ are thin and so 
$\overline{\zeta\alpha}$ generates $A(a,d)$. Thus $\zeta'\zeta$ is not a zero-path and $\rho$ not an arrow.
Thus there is an arrow $\zeta':d \rightarrow z$ because there is a long path containing $\zeta$.  We have $\overline{\zeta'\zeta} \sim \overline{\beta\rho}$ and we divide 
$A'$ by $\overline{\beta\rho}$ to get a zero-relation algebra having the obvious $\tilde{D}_{4}$ - quiver as a convex subcategory in its universal cover. The situation is shown in figure 1.

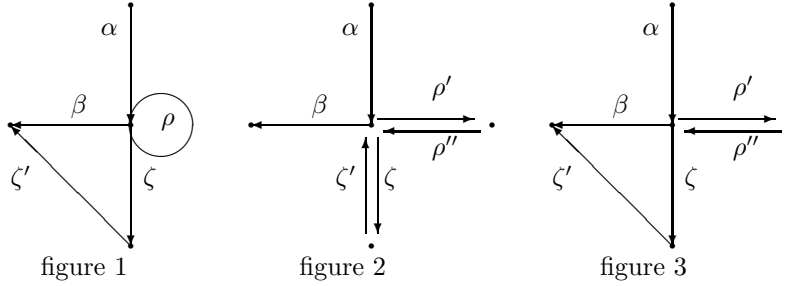
\setlength{\unitlength}{0.8cm}
\begin{picture}(15,5)

\put(0.5,-0.5){figure 1}
\put(2,4){\circle*{0.1}}
\put(2,0){\circle*{0.1}}
\put(2,2){\circle*{0.1}}
\put(0,2){\circle*{0.1}}
\put(2,2){\vector(-1,0){2}}
\put(2,4){\vector(0,-1){2}}
\put(2,4){\vector(0,-1){2}}
\put(2,2){\vector(0,-1){2}}
\put(2,0){\vector(-1,1){2}}
\put(2.5,2){\circle{1.0}}

\put(1.5,3.5){$\alpha$}
\put(1,2.2){$\beta$}
\put(2.2,1){$\zeta$}
\put(0,1){$\zeta'$}
\put(2.5,2){$\rho$}

\put(4.8,-0.5){figure 2}
\put(6,4){\circle*{0.1}}
\put(6,0){\circle*{0.1}}
\put(6,2){\circle*{0.1}}
\put(8,2){\circle*{0.1}}
\put(4,2){\circle*{0.1}}
\put(6,2){\vector(-1,0){2}}
\put(7.8,1.9){\vector(-1,0){1.6}}
\put(6.1,2.1){\vector(1,0){1.6}}
\put(6,4){\vector(0,-1){2}}

\put(6.1,1.8){\vector(0,-1){1.6}}
\put(5.9,0.2){\vector(0,1){1.6}}

\put(5.5,3.5){$\alpha$}
\put(5,2.2){$\beta$}
\put(6.2,1){$\zeta$}
\put(5.4,1){$\zeta'$}
\put(7,2.5){$\rho'$}\put(7,1.5){$\rho''$}

\put(9.8,-0.5){figure 3}
\put(11,4){\circle*{0.1}}
\put(11,0){\circle*{0.1}}
\put(11,2){\circle*{0.1}}
\put(13,2){\circle*{0.1}}
\put(9,2){\circle*{0.1}}
\put(11,2){\vector(-1,0){2}}
\put(12.8,1.9){\vector(-1,0){1.6}}
\put(11.1,2.1){\vector(1,0){1.6}}
\put(11,4){\vector(0,-1){2}}

\put(11,2){\vector(0,-1){2}}
\put(11,0){\vector(-1,1){2}}

\put(10.5,3.5){$\alpha$}
\put(10,2.2){$\beta$}
\put(11.2,1){$\zeta$}
\put(9,1){$\zeta'$}
\put(12,2.5){$\rho'$}\put(12,1.5){$\rho''$}

\end{picture}
\vspace{0.8cm}

Thus we have $m\geq 2$. We consider $\rho':=\rho_{1}:b \rightarrow c$ and $\zeta:=\zeta_{1}:b \rightarrow d$ and we assume $\zeta_{1} \neq \rho_{1}$. This time we study the full subcategory $A'$ supported 
by $a,b,c,d,z$ and its quiver $Q'$. Here $d\neq b$ because $r$ is not irreducible, $d\neq z$ because $\zeta \neq \beta$ and finally $d\neq c$ because $\zeta \neq \rho'$. Thus $Q'$ has 
$5$ points and the arrows $\alpha,\beta,\zeta,\rho'$ survive. Now $a,c,d$ are all thin so that all $A(x,y)$ between any of these points have dimension $\leq 1$. From part iii) of lemma 2 we also 
have $dim\,A(b,c)=dim\,A(b,d)=1$.
 We want to show that $A(c,d)=0$. If not there is a non-zero path $\delta:c \rightarrow d$ and so a long path $p_{2}\delta p_{1}$. We have 
$\overline{p_{1}}\sim \overline{\rho'\alpha}$ and $\overline{p_{2}\delta\rho'\alpha}\neq 0$ contradicting that $\zeta$ is an arrow. An analogous reasoning shows $A(d,c)=0$.
From $A(c,b)\neq 0$ we obtain an arrow $\rho'':c \rightarrow b$ and there is no loop at $b$. 

For $A(d,b)\neq 0$ one gets an arrow $\zeta':d \rightarrow b$ and then there is no arrow 
$d \rightarrow z$ because $\zeta'$ occurs in a long path $p_{2}\zeta'\zeta\alpha$. The quiver $Q'$ is shown in figure 2. We have $\overline{\zeta'\zeta} \sim \overline{\rho''\rho'} \sim 
\overline{\rho}$. We divide $A'$ by $\overline{\rho}$ and obtain a zero-relation algebra with a $\tilde{D}_{4}$- quiver in its universal cover.

For $A(d,b)=0$ there is an arrow $\zeta':d \rightarrow z$ because $\zeta$ belongs to a long path. The situation is illustrated by figure 3. 
We have $\overline{\zeta'\zeta}\sim \overline{\beta\rho}$. Dividing by this we end up with another zero-relation algebra with a $\tilde{D}_{4}$- quiver in its universal cover.

We have shown that $\zeta_{1}=\rho_{1}$ and we will show by induction on $i$ for $1\leq i\leq m$ that $\zeta_{i}$ exists and coincides with $\rho_{i}$. The start for the 
induction was just shown 
and we explain the step from $i-1$ to $i$. Consider $\zeta_{j}:d_{j-1}\rightarrow d_{j}$ and $\rho_{j}:c_{j-1}\rightarrow c_{j}$ ( $c_{0}=d_{0}=b$ ) for all $j \leq i-1$. 
Since $c_{i-1}=d_{i-1}\neq z$ the arrow $\zeta_{i}$ exists. Set $c:=c_{i}$ and $d:=d_{i}$. Any non-zero path $\delta:d \rightarrow c$ lies on a long path $\overline{p_{2}\delta p_{1}} \sim \overline{p_{2}\delta\zeta_{i}\rho_{i-1}\ldots \rho_{1}\alpha} \sim 
\overline{p_{2}\rho_{i}\ldots \rho_{1}\alpha}$ whence 
$\overline{\delta\zeta_{i}} \sim \overline{\rho_{i}}$ which is a contradiction. Thus $A(d,c)=0$. For $i\neq m$ i.e. $c_{i}\neq b$ one has also $d_{i}\neq b$ and one shows similarly $A(c,d)=0$.
 We look as before at the full subcategory $A'$ supported 
by $a,b,c,d,z$ and again we end up with the two cases shown in the figures 2 and 3. Argueing as above we always get a contradiction. Finally for $i=m$ we consider the full subquiver
 supported by $a,b,d,z$ and we are in the situation of figure 1 and get the same contradiction.
Thus we obtain $\zeta_{m}\ldots \zeta_{1}\alpha = \rho_{m}\ldots \rho_{1}\alpha$ and this path can only be prolongated by $\beta$.

Finally we consider a long path $p$ from $a$ to $z$ which does not start with $\alpha$. Suppose we have a proper decomposition $p=p_{2}p_{1}$ such that the end-
point $d$ of $p_{1}$ lies on $\rho$. For $d\neq b$ we have   $ dim \,A (d,z) \,=1$ and we find a subpath $\rho'$ of $\rho$ such that $\overline{p_{2}} \sim \overline{\beta\rho'}$.
Then $\beta\rho'p_{1}$ is a long path ending with $\beta$ but not starting with $\alpha$. This contradicts the dual of   part i). For $d=b$ we have 
$\overline{p_{1}} \sim \overline{\rho\alpha}$ because $\alpha$ is an arrow and then $p_{2}\rho\alpha$ is a long path starting with $\alpha$ and therefore ending with $\beta$. 
Thus $p$ ends with $\beta$ and we obtain 
the contradiction 
that $p$ starts with $\alpha$ again by the dual of part i).

\end{proof}

\begin{lemma}
 Let $\alpha':a \rightarrow b'$ be an arrow such that the interior points of all long paths starting with $\alpha'$ are thin. Then there is only one such path.
\end{lemma}
\begin{proof}
 Let $p=\zeta_{n}\ldots \zeta_{2}\zeta_{1}$ and $q=\xi_{m}\ldots \xi_{2}\xi_{1}$ be two different long paths with $\zeta_{1}=\xi_{1}=\alpha'$. 
 Because of $dim \,A(b',z)=1$ we have $\overline{p} \sim \overline{q}$. By symmetry we can assume that $n\geq m$. Then $\xi_{j}=\zeta_{j}$ for all $1 \leq j \leq m$ implies $n=m$ and $p=q$ because $z$ is a sink and so $\zeta_{m+1}$ cannot exist.
So let $j>1$ be the smallest index with $\zeta_{j}:c \rightarrow d \neq \xi_{j}:c \rightarrow e$. Then we have $A(d,e)=0$. Namely a non-zero path $\delta:d \rightarrow e$ would occur in
 a long path $p_{2}\delta\zeta_{j}\zeta_{j-1}\ldots \zeta_{1}$ and so by cancellation $\overline{\delta\zeta_{j}} \sim \overline{\xi_{j}}$ which is a contradiction. Symmetrically we have $A(e,d)=0$. Now we consider
the  full subcategory $A'$ supported by $a,z,d,e$ and its quiver $Q'$. It has arrows $a \rightarrow d$, $a \rightarrow e$, $d \rightarrow z$ and $e \rightarrow z$.
 Because of $dim \,A(a,z)=2$ and $\overline{p} \sim \overline{q}$ we 
have also
 an arrow $a \rightarrow z$. Then the separated quiver of $A'$ contains a quiver of type $\tilde{D}_{5}$.
\end{proof}

\subsection{Suspicious algebras of type 2}

\begin{proposition}
 The suspicious algebras of type 2 are exactly the algebras listed in the first four families.
 
\end{proposition}
\begin{proof} Of course all the algebras in the four families are not distributive. The algebras in the first two families are tame concealed, whence in particular minimal 
representation-infinite by proposition 6 in section 6.2.

For an algebra in one of the families 3 or 4 one has to look at quotients by a one-dimensional ideal generated by $\overline{x_{0}\beta\alpha + x_{1}\beta\rho\alpha}$. In fact, by changing
 the presentation slightly
only the values $1$
or $0$ have to be considered for $x_{0}$ and $x_{1}$. One obtains two non-isomorphic quotients for $C(p)$ and three for $D(p,q)$ which are all representation-finite by the finiteness-criterion
( see \cite[theorem 27]{BSURVEY} or section 8 ).

Reversely, let $A$ be a suspicious algebra of type 2. Observe that all points occur in a long path. Assume first that there is a thick point $b$. If there is only 
one arrow $\alpha:a \rightarrow b$ starting at $a$ we obtain an algebra of
 the family $C(p)$ 
by lemma  3. Thus let $\alpha':a \rightarrow b'$ be a second arrow where $b'\neq z$ by lemma 1. For the uniquely determined long path $p$ starting wih $\alpha'$ we have 
$\overline{p}= x_{0}\overline{\beta\alpha}+
 x_{1} \overline{\beta\rho\alpha}$. For $x_{0}\neq 0$ we change the presentation to obtain an algebra of type $D(p,q)$. For $x_{0}=0$ we look at the full subcategory $A'$ 
supported by $a,b,b',z$ and
we get an algebra  defined by  the relation $p = \beta\rho\alpha$.
Dividing out by $\overline{p}$ one obtains a zero-relation algebra containing a quiver-algebra of type $\tilde{E}_{6}$ in its universal cover. Thus $A$ is not minimal representation-infinite.
Finally there cannot be a third arrow $\alpha'':a \rightarrow b''$. Namely the full subcategory $A'$ supported by $a,b,b',b''$ is then already representation-infinite because $\tilde{A'}$ 
contains a quiver of type $\tilde{D}_{7}$.

So we can assume that there is no thick point. By lemma 4 each arrow $\alpha_{i}:a \rightarrow b_{i}$ starting at $a$ can be prolongated to a uniquely determined long path $p_{i}$ and these 
paths have no interior points in common  by the dual of lemma 4 and because all points are thin. If only two arrows start at $a$ then $\overline{p_{1}}$ and $\overline{p_{2}}$ are a basis of $S$ and $a$ is isomorphic to
an $A(p,q)$. So assume there are three arrows starting at $a$. If $\overline{p_{1}} \sim \overline{p_{2}}$ we look at the full subcategory $A'$ supported by $a,b_{1},b_{2},z$. Dividing out by 
$\overline{p_{1}}$ we obtain a zero-relation-algebra containing a quiver of type $\tilde{D}_{5}$ in its universal cover. By symmetry we can assume that for $i\neq j$ the vectors
$\overline{p_{i}}$ and $\overline{p_{j}}$ are linearly independent. Because of $dim\,S=2$  we have a relation 
$x_{1} \overline{p_{1}}+x_{2} \overline{p_{2}}+x_{3 }\overline{p_{3}}=0$
with $x_{i}\neq 0$ for all $i$.
Changing the presentation slightly we find that $A$ belongs to the family $B(p,q)$. The conditions on $p$ and $q$ follow from the fact that the full subcategory supported by all points
 except $a$ is representation-finite.

The case where more than three arrows start at $a$ is excluded because $A$ is minimal representation-infinite.

\end{proof}

\section{Algebras of type 3}

\subsection{Each point divides exactly one of the morphisms r, s or t}

Now we study suspicious algebras of type 3. We fix morphisms $s,r,t$ generating $rad \,A(a,a)$, $rad\,A(z,z)$ and $A(a,z)$ as bimodules.
A point $x$ divides a non-zero morphism $f$ if $x$ is an interior pont of a path $p$ with $\overline{p}\sim f$. 
We consider full subcategories $A'$ containing $a,z$ and a third point
$b$ that varies. The quiver of $A'$ is then denoted by $Q'$ and the  possible arrows by $\alpha_{1}:a \rightarrow b$,$\alpha_{2}:b \rightarrow a$, 
$\gamma_{1}:b \rightarrow z$, $\gamma_{2}:z \rightarrow b$, $\sigma: a \rightarrow a$, $\mu:b \rightarrow b$, $\rho:z \rightarrow z$ and $\beta:a \rightarrow z$. The situation is illustrated
 in figure 4.

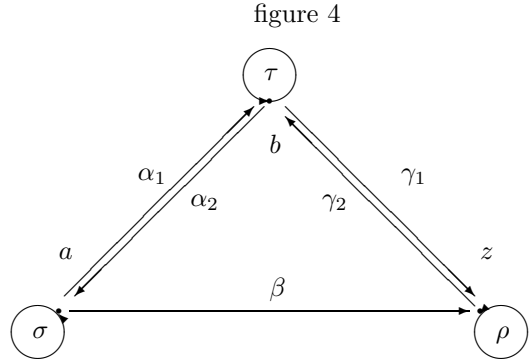
\setlength{\unitlength}{0.7cm}
\begin{picture}(15,7)

\put(5.7,5.5){figure 4}
\put(6,4){\circle*{0.1}}
\put(2,0){\circle*{0.1}}
\put(10,0){\circle*{0.1}}

\put(2.2,0){\vector(1,0){7.6}}
\put(2.1,0.3){\vector(1,1){3.6}}
\put(5.9,3.9){\vector(-1,-1){3.6}}
\put(6.3,3.9){\vector(1,-1){3.6}}
\put(9.9,0.1){\vector(-1,+1){3.6}}
\put(6,4.5){\circle{1.0}}
\put(10.4,-0.4){\circle{1.0}}
\put(2,-0.05){\vector(-1,1){0.01}}\put(10,-0.03){\vector(-1,-1){0.01}}
\put(6,4){\vector(1,0){0.01}}
\put(1.6,-0.4){\circle{1.0}}

\put(3.5,2.5){$\alpha_{1}$}
\put(4.5,2){$\alpha_{2}$}
\put(7,2){$\gamma_{2}$}
\put(8.5,2.5){$\gamma_{1}$}
\put(6,0.3){$\beta$}
\put(1.5,-0.5){$\sigma$}
\put(5.9,4.4){$\tau$}
\put(10.3,-0.5){$\rho$}
\put(2,1){$a$}
\put(6,3){$b$}
\put(10,1){$z$}

\end{picture}
\vspace{0.8cm}

The first lemma restricts the shapes of the possible quivers $Q'$.

\begin{lemma}
 Using the above notation we have:
\begin{enumerate}
 \item $Q'$ cannot contain the quiver $Q_{1}$ consisting of the  four arrows $\alpha_{1},\alpha_{2},\gamma_{1}, \gamma_{2}$.
\item $Q'$ cannot contain the quiver $Q_{2}$ consisting of $\alpha_{1}$,$\alpha_{2}$,$\gamma_{1}$ and $\rho$.
\item $Q'$ cannot contain $Q_{3}$ given by the arrows $\sigma$,$\beta$ $\rho$,$\alpha_{1}$ and $\gamma_{1}$.
\end{enumerate}

\end{lemma}

\begin{proof} Suppose $Q'$ contains $Q_{1}$. The arrow $\alpha_{2}$ is part of a long path $q\alpha_{2}p$ and so $\overline{\alpha_{2}p}$ generates 
the one-dimensional radical of $A(a,a)$ and there is no loop at $a$. Dually there is no loop at $z$.  There is also no loop at $b$ because
 otherwise the separated quiver to $A'$ contains a quiver of type $\tilde{D}_{5}$.  Thus either $Q'$ coincides with $Q_{1}$ or one has to add $\beta$. In both cases we have
 $s\sim \overline{\alpha_{2}\alpha_{1}}$ and $r\sim \overline{\gamma_{2}\gamma_{1}}$.

First we treat the case
 without $\beta$. From $0\neq ts$ we see that $\alpha_{1}\alpha_{2}\alpha_{1}$ is not a zero-path so that $A(a,b)$ is cyclic over $A(b,b)$. Dually we get that $A(b,z)$ is cyclic
 over $A(b,b)$ which is a uniserial ring whose radical is generated by $X:=\overline{\alpha_{1}\alpha_{2}}$ or by $Y=\overline{\gamma_{2}\gamma_{1}}$. Up to duality we can assume that $X$ 
is a generator. From $X\overline{\alpha_{1}} \sim \overline{\alpha_{1}}s $ we get $0=\overline{\alpha_{1}}s^{2}=X^{2}\overline{\alpha_{1}}$ and it follows that $A(a,z)$
 is generated as a vector space by $\overline{\gamma_{1}}\overline{\alpha_{1}}$ and by $\overline{\gamma_{1}}X \overline{\alpha_{1}}$ in contradiction to $ dim\,A(a,z)=3$.

Thus let $\beta$ belong to $Q'$. Then $\overline{\gamma_{1}\alpha_{1}}$ lies in $S$ and so it is annihilated on both sides by all elements in $N$. Furthermore $0\neq rt$ shows that 
$\gamma_{1}\gamma_{2}\beta$ is a non-zero path. But $\overline{\gamma_{2}\beta}$ belongs to $rad \,A(b,a)$ and so it is $\overline{\alpha_{1}}f$ or $g\overline{\alpha_{1}}$ for some elements
$f,g$ in $N$. In the first case the contradiction $0=rt$ is immediate and also for $g \sim \overline{\alpha_{1}\alpha_{2}}$. The only remaining case is 
$g \sim \overline{(\gamma_{2}\gamma_{1})^{i}\alpha_{1}}$ for
 some $i\geq 0$ and again $0=rt$ follows.

Next assume that $Q'$ contains $Q_{2}$ but not $Q_{1}$. Then $\gamma_{2}$ and $\sigma$ do not exist. If $\beta$ belongs to $Q'$ then $\tau$ does not as the separated quiver shows. 
Then $\overline{\gamma_{1}\alpha_{1}}$ lies in $S$ and 
$\overline{\beta \alpha_{2}}$ in the radical of $A(b,z)$ and so it is proportional to $\overline{\rho\gamma_{1}}$ or to $\overline{\gamma_{1}(\alpha_{1}\alpha_{2})^{i}}$ for some $i\geq 1$.
In both cases the contradiction $0 \neq ts \sim \overline{\beta\alpha_{2}\alpha_{1}} \sim 0$ follows. So we have $Q'=Q_{2}$ or one has to add  $\tau$. 

We treat the case witout $\tau$ first. We have 
$t \sim \overline{\gamma_{1}(\alpha_{1}\alpha_{2})^{i}\alpha_{1}}$ for some $i\geq 0$. From $ts\neq 0 = s^{2}$ we get $i=0$ and $ts \sim \overline{\gamma_{1}\alpha_{1}\alpha_{2}\alpha_{1}}$.
On the other hand $rt \sim \overline{\rho\gamma_{1}\alpha_{1}} \sim \overline{\gamma_{1} (\alpha_{1}\alpha_{2})^{i}\alpha_{1}}$ for some $i\geq 1$  implies $rt \sim ts$ or $rt=0$. 
Both cases are a contradiction. 

So assume finally that $\tau$ exists in $Q'$. Let $n$ be the largest natural number such that $\tau^{n}$ is not a zero-path. Then also $\gamma_{1}(\tau)^{n}\alpha_{1}$ is not a zero path. 
For $n\geq 2$ the space $A(a,b)$ is only transit and $A(b,z)$ only cotransit. We look at the full subcategory $A''$ supported by $b$ and $z$. If $n\geq 3$ we divide it by the relations
$\gamma_{1}\tau^{2}$ and by $\rho\gamma_{1}$. The remaining zero-relation algebra has an obvious Galois-covering containing a quiver of type $\tilde{E}_{6}$. Thus we get $\tau^{3}=0$.
From $ts\neq 0$ we obtain that $\alpha_{1}\alpha_{2}$ is a non-zero path and therefore proportional to $\tau^{2}$. It follows the contradiction $ts \sim rt$.

The last part is trivially excluded because the separated quiver contains a quiver of type $\tilde{D}_{5}$.

\end{proof}

\begin{lemma}
 Any point $b$ different from $a$ and $z$ divides exactly one of the morphisms $r,s, t$.
\end{lemma}

\begin{proof} We always look at the full subcategory $A'$ supported by $a,b,z$ and its quiver $Q'$ and we show first that $b$ divides at most one of the morphisms.

 If $b$ divides $s$ there is a non-zero path $p$ in $Q'$ from $a$ to $a$ with $b$ 
as an interior point and $z$ is not an interior point because of $A(z,a)=0$. Thus $\alpha_{1}$ and $\alpha_{2}$ belong to $Q'$. Dually, if $b$ divides $r$, $\gamma_{1}$ and $\gamma_{2}$ 
belong to $Q'$. 
Thus part i) of the last lemma implies that $b$ cannot divide $s$ and $r$.

If $b$ divides $s$ and $t$ but not $r$ then  $\alpha_{1}$, $\alpha_{2}$ and $\rho$ exist, but not $\gamma_{2}$ which would produce a non-zero-path from $z$ to $z$ 
saying that $b$ divides $r$. If $\gamma_{1}$ does  not exist then $\beta$ does and  $t \sim \overline{\beta}$ is true. Then  any path $p$ from $a$ to $z$ with interior point $b$ satisfies 
$\overline{p} \in S$. Thus $\gamma_{1}$ exists and the contradiction $Q_{2} \subseteq Q'$ follows.  The case that $b$ divides $r$ and $t$ is excluded by duality.

Finally, if $b$ divides none of $s,r,t$ then $Q_{3}$ is contained in $Q'$. Namely, $\alpha_{2}$ does not exist because $b$ does not divide $s$ and so $\sigma$ exists. Dually $\gamma_{2}$ does
not exist but $\rho$ does. Since $t$ does not factor through  $b$ the arrow  $\beta$ exists. Finally, there is a non-zero path from $a$ to $z$ with $b$ as an interior point because the identity at $b$ is
a non-zero path and this enforces the two arrows $\alpha_{1}$ and $\gamma_{1}$.

\end{proof}
\subsection{The uniqueness of the two cycles and the bridge}

\begin{lemma}
 Suppose $b$ divides $s$. Choose a path $p=\delta_{m}\delta_{m-1}\ldots \delta_{1}$ with $\overline{p}=s$. We consider the full subcategory $A'$ supported by $a,b,z$ and its quiver $Q'$.
 Then the following holds:
\begin{enumerate}
\item $Q'$ contains only the arrows $\alpha_{1},\alpha_{2},\beta, \rho$ defined in figure 4.
 \item $b$ is a thin point and $\overline{\delta_{1}\delta_{m}}=0$.
\item $A(a,b), A(b,a)$ and $A(b,z)$ all have dimension one.
\item Any path $q$ in $Q$ with $\overline{q}=s$ coincides with $p$.
\end{enumerate}

\end{lemma}
\begin{proof}  The arrows $\alpha_{1}$,$\alpha_{2}$,$\rho$ and $\beta$ exist because $b$ divides $s$, but neither $r$ nor $t$. 
 An additional arrow
$\gamma_{1}$ is excluded as shown in the proof of part ii) of lemma  5  and an arrow $\gamma_{2}$ implies that $b$ divides $r$.
Suppose now that $Q'$ contains also $\tau$ and let $n$ be the largest natural number such that $\tau^{n}$ is not a zero-path. Then also $\beta\alpha_{2}\tau^{n}\alpha_{1}$ is not a zero path. 
 We look at the full subcategory $A''$ supported by $b$ and $z$ and its  quiver $Q''$ that  contains the two loops $\tau$ and $\rho$ and one arrow $\epsilon$ from $b$ to $z$ induced from
 $\overline{\beta\alpha_{2}}$. Because $\overline{\rho\beta}$   lies in $S$ we have  $\overline{\rho\epsilon} =0$.  For  $n\geq 2$ we find in the Galois-covering of $A''$  a quiver of 
type $\tilde{E}_{6}$. Now $A''$ is already a full subcategory of the proper quotient of $A'$ by $\overline{\beta\alpha_{2}\tau^{n}\alpha_{1}}$ and therefore $A''$ is representation-finite.
 Thus $\tau^{2}$ is a 
zero-path and so is $\alpha_{1}\alpha_{2}$. We can arrange by a slight change of the presentation that in addition  $\alpha_{2}\alpha_{1}$ is a zero path. Then 
$A'/\langle \beta\alpha_{2}\tau \rangle$ is a special biserial
 algebra containing  the cyclic word $\tau^{-1}\alpha_{1}(\rho\beta)^{-1}\alpha_{2}\beta$. Thus $A'$ is not minimal representation-infinite. This contradiction shows that $Q'$ has only 
four 
arrows.

If $b$ is not thin $\alpha_{1}\alpha_{2}$ is not a zero-path and so it can be prolongated to a non-zero path $\beta\alpha_{2}\alpha_{1}\alpha_{2}\alpha_{1}$ 
 contradicting $s^{2}=0$.
 If $\delta_{1}\delta_{m}$ is not a zero-path we can prolongate it to a non zero-path $\delta_{1}\delta_{m}p_{1}$ and 
so the
 end-point of $\delta_{1}$ is not thin.

 To see that $dim A(a,b)$, $dim A(b,a)$ and $ dim\,A(b,z)$ are $1$ we can assume that $b$ is an interior point of $p$.
For $dim \,A(a,b)\geq 2$ the space $A(a,b)$ is cyclic over $A(a,a)$ and so  $fs\neq 0$ where $f \sim q$ for the subpath $q$ of $p$ leading from $a$ to $b$ and where $s \sim p$. This 
contradicts the fact that  $\alpha_{1}\alpha_{m}$ is a zero-path. The proof for $A(b,a)$ is similar. Finally  $A(b,z)$ is 
generated by $\overline{\beta\alpha_{2}}$.

Part iv) is clear for $m=1$. So suppose $q=\delta'_{n}\delta'_{n-1}\ldots \delta'_{1}$ is another path with $\overline{q}\sim s$. Then we have also $n>1$ and because $a$ is not an interior 
point of $p$ or $q$ there is a smallest index $i$ with 
$\delta_{i}\neq \delta'_{i}$. Let $b\neq b'$ be the ending points of $\delta_{i}$ and $\delta'_{i}$ and let $c$ be the starting point. We decompose $p$ and $q$ as $p=p_{2}\delta_{i}p_{1}$ and 
$q=q_{2}\delta'_{i}p_{1}$. For $b=a$ we obtain $\delta_{i} \sim q_{2}\delta'_{i}$ which is impossible. The same is true for  $b'=a$. We claim that $A(b,b')=0$.
If not then  $A(b,b')=k$ 
because $A(b,b')$ is uniserial. Thus there
 is a path $q':b \rightarrow b'$ such that $\overline{q'}$ is a basis of $A(b,b')$ and this path can be prolongated to a path $q'p''$ from $a$ to $b'$ which is  is not a zero-path. Since 
$A(a,b')$ and $A(a,b)$ have dimension  we obtain that 
$q'p'' \sim q'\delta_{i}p_{1} \sim \delta'_{i}p_{1}$. Thus $q'\delta_{i}$ and $\delta'_{i}$ are two non-zero path between $c$ and $b'$ where $A(c,b')$ has dimension one for $c=a$ 
by 
part iii) 
 and 
also for $c$ dividing $s$ because $A(c,b)$ is uniserial. This implies that $\delta'_{i}$ is not an arrow, a contradiction. Thus we have $A(b,b')=0$ and symmetrically $A(b',b)=0$. 
Now consider  the full subcategory $A'$ supported by $a,b,b',z$ and its quiver $Q'$. Because $A(b,a)\neq 0$ but $A(b,b')=0$ and $A(z,b)=0$ there  is an arrow $b \rightarrow a$. Symmetrically there is an arrow
$b' \rightarrow a$. Similarly we have $A(a,b)\neq 0$ but $A(b',b)=A(z,b)=0$ leading to an arrow $a \rightarrow b$. Again by symmetry we also have an arrow $a \rightarrow b'$. Of course also
 $\beta:a \rightarrow z$ and $\rho:z \rightarrow z$ belong to $Q'$. Now we divide $A'$ by the non-trivial ideal $I_{a}$ to obtain a zero-relation algebra where the path $\rho\beta$ is not killed. In the 
universal cover we find a quiver of type$\tilde{D}_{7}$.

\end{proof}

\begin{lemma}
 Suppose $b$ divides $t$. Choose a path $p=\beta_{m}\beta_{m-1}\ldots \beta_{1}$ with $\overline{p}\sim t$. Then we have:
\begin{enumerate}
 \item $b$ is a thin point and $A(b,a)=A(z,b)=0$.
\item Any path $q$ with $\overline{q}\sim t$ coincides with $p$.
\end{enumerate}

\end{lemma}
\begin{proof}
 As usual we look at the full subcategory $A'$ supported by $a,b,z$ and its quiver $Q'$ and we use the notations from figure 4. There is no arrow $\alpha_{2}$ because $b$ does not divide $s$ 
and dually there is no arrow $\gamma_{2}$. Furthermore $\beta$ does not belong to $Q'$ since $b$ divides $t$.
Thus $\sigma$, $\rho$, $\alpha_{1}$ and $\gamma_{1}$ exist. Suppose that there is a loop $\tau$ in addition. If $\tau^{3}$ is not a zero-path also $\gamma_{1}\tau^{3}$ is none and we have 
$\overline{\rho\gamma_{1}}=\overline{\gamma_{1}( x_{2}\tau^{2}+x_{3} \tau^{3}+ \dots )}$. In the fullsubcategory  $A''$ supported by $b,z$ and its quiver $Q''$ we introduce the relations
 $\gamma_{1}\tau^{2}$ and $\rho\gamma_{1}$. Then $\tau^{3}$ is not annihilated and we find in the universal cover of the resulting zero-relation algebra a quiver of type $\tilde{E}_{6}$.
Thus  $\tau^{3}$ is a zero-path. We obtain in $A'$ the contradiction
 $\overline{\gamma_{1}\alpha_{1}\sigma} \sim \overline{\gamma_{1}\tau^{i}\alpha_{1}}\sim \overline{\rho\gamma_{1}\alpha_{1}}$ where $i=2$ if $\tau^{2}$ is not a zero path and $i=1$ in the other case. The proof of part i)
 is complete.

Let $q=\beta'_{n}\beta'_{n-1}\ldots \beta'_{1}$ be another path with $\overline{q}\sim t$. Let $i$ be the smallest index with $\beta_{i}\neq \beta'_{i}$ and write $p=p_{2}\beta_{i}p_{1}$
and $q=q_{2}\beta'_{i}p_{1}$. Let $c$ be the starting point of $\beta_{i}$ and $\beta'_{i}$ and let $b,b'$ be the two different end points. We consider the full subcategory $A'$ 
supported by $a,b,b',z$ and its quiver $Q'$ and we claim that it contains arrows $\zeta:a \rightarrow b$ and $\zeta':a \rightarrow b'$. This is clear if $p_{1}$ has length $0$. 
Thus assume $c\neq a$.
If the morphism $\overline{\beta_{i}p_{1}}$ does not induce an arrow in $Q'$ then there is a path $\xi:b \rightarrow b'$ in $Q$ with 
$\overline{\xi\beta'_{i}p_{1}} \sim \overline{\beta_{i}p_{1}}$ because $A(b,b')$ has dimension one at most as a uniserial bimodule over $k$. This is also true for $A(c,b')$ and so we get the contradiction
$\overline{\beta_{i}} \sim \overline{\xi\beta'_{i}}$. Thus $Q'$ contains the arrows $\sigma$, $\zeta$ and $\zeta'$. 

We claim that $A(b,b')=0$. If not there is an arrow   $\xi:b \rightarrow b'$ in $Q'$ that can be prolongated to a path $q$ with non-zero $\overline{q} \in S$. Then $\xi\zeta$ is not a zero-path 
and we find $\overline{\xi\zeta} \sim \overline{\zeta'\sigma}$ because $\zeta'$ is an arrow. From  $\overline{q_{2}\xi} = \overline{x\rho p_{2} + yp_{2}}$ and $\overline{q_{2}\xi\zeta\sigma}=0$ we obtain the 
contradiction $\overline{\beta\sigma} \sim \overline{q_{2}\zeta'\sigma} \sim \overline{q_{2}\xi\zeta} \sim \overline{\rho p_{2}\zeta} \sim \overline{\rho\beta}$. Symmetrically we get $A(b',b)=0$. 

Finally the full subcategory$A''$ supported by $a,b,b'$ is a zero-relation algebra having a quiver of type $\tilde{D}_{5}$ in its universal cover.
\end{proof}

\subsection{Suspicious algebras of type 3}
\begin{proposition}
 The suspicious algebras of type 3 are exactly the algebras $E(p,q,r)$.
 
\end{proposition}
\begin{proof}
 For an algebra $A$ in the fifth family define $p=\beta_{r}\ldots \beta_{1}\alpha_{q}\ldots \alpha_{1}$ and $q=\gamma_{p}\ldots \gamma_{1}\beta_{r}\ldots \beta_{1}$. Then $A$ is a special
 biserial algebra with $q^{-1}p$ as the only  primitive cyclic word up to inversion and cyclic permutation. Any proper quotient is still special biserial but without any cyclic word and 
so it is representation-finite.

Reversely, let $A$ be an algebra of type 3 with quiver $Q$ and let 
  $\alpha= \alpha_{q}\ldots \alpha_{1}$, $\beta= \beta_{r}\ldots \beta_{1}$, $\gamma=\gamma_{p}\ldots \gamma_{1}$ be the three uniquely determined paths giving $s,t,r$. Since all 
interior points of $\alpha$  are thin  by lemma 7 the interior points are pairwise different. The same holds  for $\beta$ by lemma 8 and for $\gamma$ by the dual of lemma 7. Furthermore 
the union of the interior points is disjoint by lemma 6 and $Q_{0}$ consists in these interior points and $a$ and $z$.
We 
show that any arrow in $Q$ occurs already in one of the three paths. 
So let $\phi:x \rightarrow y$ be an  arrow.
First take $x=a$.  For $y=a$ resp. $y=z$ we get $q=1$ and $\phi=\alpha_{1}$ resp. $r=1$ and $\phi= \beta_{1}$. If $y$ is thin, there is always a non zero-path from $a$ to $y$ and 
we always have $dim \,A(a,y)=1$. Thus $\phi$ is $\alpha_{1}$ or $\beta_{1}$. Next we look at $x=z$ and a thin point $y$. Then we have $dim \,A(z,y)=0$ if $y$ divides $s$ or $t$ 
and $dim \,A(z,y)=1$
if $y$ divides $r$. Thus only $\phi=\gamma_{1}$ is possible. Thus there is no additional arrow starting in a thick point. 
By duality we can assume now that $x$ and $y$ are thin.

We consider always a prolongation $q\phi p$ of $\phi$ such that $\overline{q\phi p} \in S$. First assume that $x$ divides $s$ and also $y$. 
Let $x$ be the endpoint of $\alpha_{i}$ and $y$ the endpoint of $\alpha_{j}$. For $i>j$ we can assume that $p=\alpha_{i}\alpha_{i-1}\ldots \alpha_{1}$ because of $dim \,A(a,x)=1$ 
and then the non-zero path $\phi p$ runs 
twice through $y$ contradicting the fact that $y$ is thin. For $i < j$ we have $\overline{\phi} \sim \overline{\alpha_{i}\ldots \alpha_{j+1}}$ whence there is an arrow only for $i=j+1$.
Next suppose $y$ divides $t$. Then there is a non-zero path of length $\geq 2$ from $x$ to $y$ in $Q'$ and because of $dim\, A(x,y)=1$ there can be no arrow $\phi$.
Finally assume that $y$ divides $r$. Then we get $\overline{\phi p}\sim r't$ and $\overline{q\phi} \sim ts'$ for some non-zero morphisms $r' \in A(z,y)$ and $s' \in A(a,x)$.
The contradiction $ts \sim rt$ follows. 

Next assume that $x$ is the ending point of $\beta_{i}$ for some $i$. If $y$ divides $s$ or if it is the ending point of $\beta_{j}$ with $j<i$ then there is a long path $q\phi p$ 
running twice through 
the thin point $x$ which is a contradiction. For any other $y$ we have a non-zero-path from $x$ to $y$ and so there is an arrow only if $\beta_{i+1}$ ends in $y$. Up to duality the only 
remaining case is when $x$ edivides $r$ and $y$ divides $s$. This would give a long path running first through $z$ and then through $a$ which is excluded by $A(z,a)=0$.

We
 have determined the quiver of $A$. We know already that the two zero-relations $\alpha_{1}\alpha_{m}$ and $\gamma_{1}\gamma_{p}$ hold in $A$. If $\gamma_{1}\beta_{r}\ldots
\beta_{1}\alpha_{m}$ is not a zero-path it can be prolongated to a long path  contradicting $rts=0$.
\end{proof}

\section{Two consequences of theorem 1}

\subsection{Accessible modules for non-distributive algebras}

Ringel defined in \cite{Ringel2}  the notion of an accessible module of finite length: To start with all modules of length $1$ are accessible and a module of length $n\geq 2$ is accessible 
if it is indecomposable 
and if it has an accessible submodule or quotient of length $n-1$. It is known since a long time \cite{Extensions,Degenerations} that  any indecomposable is accessible provided that the field is algebraically closed and $A$ is representation-finite 
 or tame concealed.
 Ringel has shown in \cite{Ringel2} the next result which follows also from theorem 1..

\begin{theorem}
 Let $A$ be a finite-dimensional algebra over an algebraically closed field. If there is an indecomposable of length $n$ there is also an accessible of length $n$.
\end{theorem}
\begin{proof}
 We can assume that $A$ is minimal representation-infinite and basic and we have to show that accessible modules exist in all dimensions. 
The case when $A$ is distributive is  the difficult one and this case is treated in \cite{Bind} without mentioning the new terminus 'accessible'.

For a  non-distributive algebra  Ringel has given in \cite{Ringel2} a nice direct argument. Of course  one can now alternatively inspect  the list 
in theorem 1. It suffices to look at the 'unglued' algebras.
The first two families consist of tame concealed algebras and then all indecomposables are accessible. The same is true by \cite{Rinmin} for the last family containing only special biserial
 algebras.
In the remaining two cases there is  an obvious Galois-cover  with fundamental group ${\bf Z}$ that contains a tame-concealed algebra $B$ of type $\tilde{D}_{n}$ as a convex subcategory. The
 push-downs of the indecomposable $B$-modules provide accessibles  in each dimension.
\end{proof}

\subsection{The proof of theorem 2}

Fix a natural number $d$. We want to find a finite list $L$ of ${\bf Z}$-algebras $A$ such that for any algebraically closed field $k$ the 
extended algebras $A\otimes k$ are a representative system of isomorphism classes
of basic minimal representation-infinite algebras of dimension $d$. We consider first the non-distributive algebras.

The relations imposed on any of the quivers $Q$ occurring in theorem 1 and also their glued versions  make sense already in ${\bf Z}Q$ and the quotient algebra $A$ 
is always a free ${\bf Z}$-module. Define $L_{1}$ as the set of 
${\bf Z}$-algebras obtained that way which are free of rank $d$. By scalar extension one obtains for all fields a list of non-distributive minimal
 representation-infinite algebras.

To treat the distributive algebras  we need some concepts and highly non-trivial results all described in \cite{BSURVEY}. We choose any algebraically closed field $k$. There is only a finite list $L_{2}'$ of 
equivalence classes of ray categories $P$ such
 that the linearization $kP$ is minimal 
representation-infinite of  dimension $d$.
By the finiteness-criterion  this property is independent of the chosen field. Furthermore 
any basic distributive minimal representation-infinite algebra is isomorphic to the linearization of a ray category.
Finally $kP$ and $kP'$ are isomorphic if and only if $P$ and $P'$ are equivalent categories. We take $L_{2}$ as the finite set of algebras ${ \bf Z}P$ with $P$ in $L_{2}'$ and define $L$ as the union 
of $L_{1}$ and $L_{2}$.

\section{Tame concealed and  critical simply connected algebras}

\subsection{Three notions of minimality}

For an algebra $A$ of infinite representation type we can ask whether all proper quotients $A/I$ are representation-finite resp. only quotients $A/\langle e_{x} \rangle$ for an arbitrary point $x$ in $Q$ resp.
  only  quotients $A/\langle e_{x} \rangle$ for  a source or a sink $x$ in $Q$. In this way we obtain the set $\mathcal{M}$ of ( isomorphism classes of ) minimal representation-infinite algebras 
resp. the set $\mathcal{A}$ of 
representation-infinite algebras such that almost all  indecomposables ( up to isomorphism ) are sincere resp. the set $\mathcal{B}$ of 
representation-infinite algebras such that almost all indecomposables are extremal or - equivalently -  such that all proper convex subcategories are representation-finite. 

We call algebras in  $\mathcal{B}$  critical. 
Sometimes e.g. in \cite{Listebild,Tame,Skowronskisimson} algebras in 
$ \mathcal{A} $ are already called minimal representation-infinite.

The  inclusions
$\mathcal{M} \subseteq \mathcal{A} \subseteq \mathcal{B}$  are proper but  restricted to algebras having a simply connected component all three sets coincide. 
First we note the following:

\begin{lemma}
 Suppose  a critical algebra $A$ has  a preprojective component $Z$. Then $Z$ contains all indecomposable projectives but no injective. Almost all
 indecomposables in $Z$ are sincere and so
$A$ is tilted from a path-algebra $kK$. If $K$ is Euclidean then $A$ is tame concealed.
\end{lemma}
\begin{proof}If $Z$ contains an indecomposable injective $De_{x}A$ it contains such a module for a sink $x$. Only the finitely many predecessors in $Z$   of this module can be extremal and so $A$ is not critical. Thus $Z$ contains no injective and so it is infinite with almost all modules extremal whence sincere by \cite{quadratic}.  Thus $Z$ contains a section $K$ and $A$ is tilted from $kK$ by a tilting module $T$ by \cite{Tame}. For tame $K$ only a preprojective ( or preinjective ) tilting module leads to a critical algebra $A$ as explained in section 4.2 (8)  in Ringels classical book  \cite{Tame}.

\end{proof}

\subsection{Tame concealed algebras are minimal representation-infinite}

I observed this already in \cite{liste}  but   it went unnoticed. Here is the proof.
\begin{proposition}  Any tame concealed algebra $A$ is minimal representation-infinite. 

\end{proposition}
\begin{proof} 
Let $I$ be a non-zero ideal in $A$ and take a non-zero  element $a$ in the intersection of $I$ and the radical.
Suppose $a$ lies in $e_{y}Ae_{x}$. The multiplication with $a$ from the right induces a minimal projective resolution
$$Ae_{y} \rightarrow Ae_{x} \rightarrow C \rightarrow 0$$ with an indecomposable  module $C$. This induces by general facts used in the existence proof for almost split sequences 
( see \cite[section 1.3]{Gabaus} \cite{Vossieck} ) after 
dualizing an exact
 sequence of functors 

$$ 0 \rightarrow Hom(C,\,) \rightarrow Hom(Ae_{x},\,) \rightarrow Hom(Ae_{y}, \, ) \rightarrow DHom( \, ,DTrC) \rightarrow 0.$$ 

Here we plug in a sincere indecomposable $U$ which we identify with the corresponding representation. Then we obtain an exact sequence 

$$ 0 \rightarrow Hom(C,U) \rightarrow U(x) \rightarrow U(y) \rightarrow DHom( U,DTrC) \rightarrow 0,$$ 
where the middle arrow is just the multiplication map $U(a):U(x) \rightarrow U(y)$.

 Thus we see that 
$U(a)=0$ impies $ Hom(C,U)\simeq  U(x)\neq 0$ and $0\neq  U(y)\simeq  Hom(U,DTrC)$ and so  
 $U$ and $C$ are regular  in the same tube $T$. Let $h$ be the number of regular simples in $T$. Their dimension vectors add up to a vector  where all components are different from zero. Since $A$ is directed we have $dim \,C(x)=1$ and so the regular length of $C$ is less than $2h$. This implies
$dim\,Hom(C,U) \leq 2$ in the uniserial category $T$. Thus we have $dim\,U(x) \leq 2 $ which shows that the regular length of $U$ is less than $3h$. 

We have shown that $a$ annihilates only finitely many sincere indecomposables ( up to isomorphism ). But by the easy implication of theorem 2 in \cite{Listebild} almost all indecomposables are sincere.

\end{proof}

\subsection{The relations between tame concealed and critical simply connected algebras}

In the criterion of \cite{Criterion} to decide whether an algebra $A$ is representation-finite I consider a  directed category $\tilde{A}$ as a Galois-covering of $A$ such that each finite subcategory $C$ has a finite convex hull $\hat{C}$ with a simply connected component in its Auslander-Reiten quiver. Such a component is given by a graded tree $(T,g)$ as in \cite[section 6]{Coverings}.  Here $\tilde{A}$ is locally representation-finite iff $A $ is representation-finite by \cite{Galoiscovering}.

 If there is a finite subcategory of $\tilde{A}$ of infinite representation type one gets just by throwing away successively extremal points     a critical convex subcategory $\hat{C}$ which has then exactly one simply connected component in its Auslander-Reiten quiver. 
  Typical examples for such a critical category are the algebras $\tilde{D}(p,q,r)$ of figure 5.
 Here  the  left hand side is a commutative diagram that disappears for $p=0$  and
 the analogous statement  holds for $r=0$ on the other side. 

\setlength{\unitlength}{0.8cm}
\begin{picture}(13,7)\put(6.5,1){figure 5}
\put(3,6){$a$}
\put(2,5){$c_{1}$}
\put(2,4){$c_{2}$}

\put(2,2){$c_{p}$}
\put(3,1){$b$}
\put(4,3.5){$z_{1}$}

\put(5.5,3.5){$z_{2}$}
\put(10,3.5){$z_{q}$}
\put(11,6){$d$}
\put(12,5){$f_{1}$}
\put(12,4){$f_{2}$}
\put(12,2){$f_{r}$}
\put(11,1){$e$}
\put(3.2,5.8){\line(2,-5){0.6}}
\put(10.2,3.3){\line(2,-5){0.6}}
\put(3.2,1.3){\line(2,5){0.6}}
\put(10.2,3.8){\line(2,5){0.6}}
\put(4.6,3.6){\line(1,0){0.6}}
\put(2.9,5.85){\vector(-1,-1){0.5}}
\put(11.9,1.85){\vector(-1,-1){0.5}}
\put(2.2,1.8){\vector(1,-1){0.5}}
\put(11.2,5.8){\vector(1,-1){0.5}}
\put(2.1,4.8){\vector(0,-1){0.3}}
\put(12.1,4.8){\vector(0,-1){0.3}}

\multiput(6.5,3.6)(0.2,0){14}{\circle*{0.03}}
\multiput(2.1,3.6)(0,-0.2){6}{\circle*{0.03}}
\multiput(12.1,3.6)(0,-0.2){6}{\circle*{0.03}}
\end{picture}

For the critical subcategories I proved in \cite{liste}:
\begin{theorem} A critical subcategory $\hat{C}$ contained in $\tilde{A}$ has  a simply connected component given by a Euclidean tree  $T$. For  $T=\tilde{D}_ {n}$ only the algebras $\tilde{D}(p,q,r)$ and $B(n-3,1)$ are possible.

\end{theorem}
Thus the critical algebras are all tame concealed by  lemma 9.
Happel and Vossieck have determined  in \cite{Listebild}  the tame concealed algebras by looking at the possible preprojective tilting modules up to $TrD$- translation  and by calculating their endomorphism algebras. The `frames' of these are depicted in the  HV-list:

\begin{theorem} An algebra $A$ is tame concealed iff it occurs in the HV-list.
\end{theorem}

They proved also in\cite{Listebild}:
\begin{theorem} Let $A$ be a basic connected algebra  of finite dimension. Then $A$ is tame concealed or a generalized Kronecker algebra iff $A$ has an infinite preprojective component and $A/AeA$ is representation-finite for each idempotent $e\neq 0$.
\end{theorem}
 Thus the tame concealed  algebras are exactly  the basic minimal representation-infinite algebras with a preprojective component.

In the literature there are many false or misleading statements concerning the last three theorems and the finiteness criterion. Here are some comments on this. 

\begin{itemize}
\item   Theorem 5 was presented by Happel at a conference in september 1982 at Luminy. The case $\tilde{A}_{n}$ is easy, $\tilde{D}_{n}$ affords some combinatorial considerations and the remaining three cases a computer and some drawing. 
Theorem 6 was not mentioned in Happels talk.
\item At the same conference theorem 4 was presented in my talk as a part of a general finiteness criterion. For the exceptional cases $\tilde{E}_{n}$ I had only the gradings of the representation-infinite algebras  as unreadable computer-lists.  So my condition was that all convex subcategories with at most $9$ points have to be representaton-finite. Up to the action of the Galois-group only finitely many convex subcategories have to be considered.

 My lists contained also some algebras which are not  critical. In \cite{liste} I removed these and then my lists are in accordance with the HV-list. The obvious lemma 9 was overlooked by me at that time and it is still not mentioned  in \cite{Br} that contains several incorrect statements.

\item Theorem  6 is published in    the article \cite{Listebild} submitted in november 1982 .  To prove the difficult direction a result of Ovsienko on  quadratic forms \cite{Ovsienko} and tilting theory are used.

 The  proof of theorem 6 is much more elegant and shorter  than my bare-handed proof of theorem 4 using only the inductive method from \cite{Coverings}.   Both proofs are completely different contradicting  \cite[chapter XIV]{Skowronskisimson}.

However, for the algebras occurring in the finiteness criterion my result is better than theorem 6 because convex subcategories are easy to detect. Also the case $\tilde{D}_{n}$ is much easier to analyze with my inductive method and this could have been used in \cite[chapter XIV]{Skowronskisimson} where on 14  pages only the cases $n\leq 6$ are treated.

\item In July 1982 at a meeting in Bielefeld I had mentioned in a private conversation with Ringel theorem 4 that I just found before.   The result was not known to him but he conjectured that there is a connection to Ovsienkos result that was unknown to me.  

\item  What Ringel writes in \cite{RinVos} about the importance of the HV-list is not true. Neither the proof of BT2 nor  the article on multiplicative bases depend on the HV-list. 
 \item It is aggravating that the finiteness criterion  \cite{Criterion} and its further developments in \cite{BSURVEY} are not mentioned in the recent literature. I will say  more about this in a  forthcoming survey on representation-finite selfinjective algebras, coverings and so on.

\end{itemize}

\section{On the classification}
\subsection{Glueings}

In section 7 we freely use ray categories ( \cite{BGRS} ) and their properties as surveyed in \cite{BSURVEY}. Thus let $A$ be a basic distributive minimal 
representation-infinite algebra. 
 By the important theorem 2 of \cite{Bind} $A$ is isomorphic to the linearization $kP$ of its associated ray category $P$,
the universal cover 
$\pi:\tilde{P} \rightarrow P$ has a free fundamental group, $\tilde{P}$ is interval-finite  and any finite subset of $\tilde{P}$ lies in a finite full convex
 subcategory $C$  of $\tilde{A}=k\tilde{P}$ 
which has a simply connected preprojective component in its module category. Here $C,\tilde{A}$ and $A$ are standard and so they admit a
 presentation induced by zero-paths and contours.

Now we  take an embedding $i:C \rightarrow \tilde{A}$ of a finite full convex  subcatgory $C$  and the composition  $p:C \rightarrow A$  with 
 $k\pi: \tilde{A} \rightarrow A$. We denote by $i$ and $p$ also the induced morphisms at the level of the quivers $Q_{C},Q_{\tilde{A}}$ and $Q_{A}$ and their
 path-categories.
We obtain an equivalence relation $R_{p}$ on the point set $(Q_{C})_{0}$ of $Q_{C}$ having the non-empty fibres of $p$ as the equivalence classes.

More general for any equivalence relation $R$ on $(Q_{C})_{0}$ we have the quotient quiver $Q_{C}/R$ with the equivalence classes as points and 
the natural surjective quiver-morphism
$q:Q_{C} \rightarrow Q_{C}/R$ extending to a functor beteen the path-categories again denoted by $q$.
We define the glued algebra $C_{R}$ as the quotient of $k(Q_{C}/R)$ by the ideal generated by the path $qv$ where $v$ is a  zero-path in $C$, by the differences $qu-qw$
 where $(u,w)$ is a contour in $C$ and by the paths $v$ in $Q_{C}/R$ that have no lifting in $Q$. Observe here that the zero-paths and the contours in $C$ are
 just those of
$\tilde{A}$ that lie in the full convex subcategory $C$. 

\begin{proposition}
 Keeping these notations and assumptions let $M$ be an $\tilde{A}$-module with support $C$ and push-down $N$.
Assume  that the powers $N^{n}$ have infinitely many pairwise non-isomorphic indecomposables as quotients. Then $A$ is isomorphic to the glued algebra $C/_{R_{p}}$.
\end{proposition}

\begin{proof}
$N$ is faithful because $A$ is minimal representation-infinite and the powers of $N$ have infinitely many pairwise non-isomorphic indecomposable quotients. 
Let $v$ be a path in $Q_{A}$ from $x$ to $y$ and let $v'$ be a lifting in $Q_{\tilde{A}}$ from $x'$ to $y'$. Since $C$ is convex $v'$ is a path in $C$ if and only if 
$x'$ and $y'$ both belong to $C$. By the definition of the push-down $N(v)N(x) \rightarrow N(y)$ acts 'diagonally' through the various liftings $M(v'):M(x') \rightarrow M(y')$.

Thus $N(v)=0$ unless there is a lifting $v'$ in $C$. In particular paths of length $0$ or $1$ have a lifting in $C$ and so the morphism from the quiver of C 
to the quiver of $A$ is surjective and it identifies $Q_{A}$ with $Q_{C}/R$.

If $v$ has no lifting in $C$ it annihilates $N$ and so it is a zero-path. If $v$ has a lifting  $v'$ in C. Then $v$ is a zero-path in $A$ if and only if $v'$ is a zero-path 
in $\tilde{A}$ if and only if $v'$ is a zero-path in $C$.

Similarly for a contour $(u,w)$  in $A$ the path $u$ is not a zero-path and so $N(u)\neq 0$ implies that there is a lifting $u'$ in $C$ with starting point $x'$.
 Then the lifting $w'$ of $w$ starting in $x'$ also lies in $C$ and $(u',w')$ is a 
contour in $C$ mapping to the given contour.

\end{proof}

\subsection{Towards the classification}

A line  $L$ of length $e$ in $\tilde{P}$ is a convex subcategory living on a linear subquiver  $x_1 \rightarrow x_{2}  \ldots  x_{e-1}\leftarrow x_{e}$  without any 
relation. The line is called critical  if $\pi(x_{1})=\pi(x_{e})$ are both sources or both  sinks in $L$ and if  $\pi(x_{2})\neq \pi(x_{e-1})$ holds. 

\begin{proposition}We keep all the assumptions and notations and we assume that  $d:= dim\,kP < \infty$. Then we have:

\begin{enumerate}
 \item Any line $L$ in $\tilde{P}$ of length $2d+1$ contains a critical line as a subline.
\item For any  critical line $L$ of length $e$ the push-down $N$ of the indecomposable $\tilde{P}$-module $M$ with support $L$ has infinitely many pairwise 
non-isomorphic quotients of dimension $e-1$.
 
 \end{enumerate}
\end{proposition} 

\begin{proof} Up to duality we can assume that $x_{1}$ is a sink in $L$ which we  write down  thereby marking all sinks $s_{1},s_{2},\ldots s_{r}$. We obtain

$$  x_{1}=s_{1}\leftarrow \ldots \rightarrow s_{2}\rightarrow \ldots \rightarrow s_{i}\leftarrow \ldots \rightarrow s_{r}\leftarrow \ldots x_{2d+1},$$
where $s_{r}=x_{2d+1}$ is possible. Thus we get $2d+1 \leq \sum_{i=1}^{r} \,dim\,Dk\tilde{P}(\,,s_{i})$ and so three sinks are mapped onto the same point under $\pi$. We can assume that these points are
 $s_{1},s_{i}$ and $s_{r}$. There is a critical subline with two of these as extremal points. The first assertion is proven and the second is shown 
in lemma 3.2 of \cite{Criterion}.

\end{proof}

Now we can show that any minimal representation-infinite algebra $A$ is isomorphic to a glued algebra $C_{R}$ where $C$ is a critical line or a critical algebra and $R$
 an equivalence relation. For the case of triangular algebras and for the  notion of minimality defining the algebras in $\mathcal{A}$ a slightgly weaker statement
 was obtained in 
\cite{Ca}.     .

We distinguish three cases depending on the structure of the universal cover $\tilde{P}$ which is not locally representation-finite by \cite{Galoiscovering}.\vspace{0.5cm}

Case I: Each finite full subcategory is representation-finite. 

Then there are indecomposable $\tilde{A}$-modules with arbitrarily large support $B$ which is always a convex full subcategory. Therefore $B$ belongs to the list 
LSS and we find a critical line $C$ with part i) of proposition 8. Part ii) and proposition 7 show that $A$ is glued from $C$ by an appropriate $R$.\vspace{0.5cm}

Case II: $\tilde{A}$ contains a critical algebra $C$ of type $\tilde{D}_{n}$, but none of type $\tilde{E}_{n}$.

Then we take for $M'$ a progenerator of $C$. The powers of $M'$ have infinitely many pairwise non-isomorphic indecomposable modules and the same holds 
for the extension
$M$ of $M'$ by zero and its push-down $N$ by basic properties of the push-down functor. Thus $A$ is  a glued algebra by proposition 7. The same argument applies in 
the last case.\vspace{0.5cm}

Case III: There is a critical convex full subcategory $C$ of type $\tilde{E}_{n}$.\vspace{0.5cm}

Unfortunately for a given $C$ there are many equivalence relations $R$ such that $C_{R}$ is not minimal representation-finite even if we restrict to those $R$ such that
the induced morphism $q:Q_{C} \rightarrow Q_{C}/R$ is injective on arrows with a common source or sink as $p:Q_{C} \rightarrow Q_{A}$ is. 

For example let $C$ be the 
quiver-algebra of the quiver shown in figure 6. There are 53 isomorphism classes of proper glueings but only 9 of them are minimal representation-infinite.
 The smallest of these algebras
has two points $x,y$ and two arrows $\alpha: x\rightarrow y$, $\beta:x \rightarrow x$ subject to the relations $\beta^{4}=\alpha\beta^{3} =0$. This algebra is minimal 
represenation-infinite but wild since its universal cover contains a hyper-critical quiver algebra.

\setlength{\unitlength}{0.8cm}
\begin{picture}(13,7)\put(6.5,1){figure 6}
\put(3,6){$a$}

\put(3,1){$b$}
\put(4,3.5){$z_{1}$}

\put(10,3.5){$z_{2}$}
\put(11,6){$d$}

\put(11,1){$e$}
\put(3.2,5.8){\vector(1,-2){0.8}}
\put(10.2,4){\vector(1,2){0.8}}
\put(4,3.3){\vector(-1,-2){0.8}}
\put(10.2,3.3){\vector(1,-2){0.8}}
\put(4.6,3.6){\vector(1,0){4.5}}

\end{picture}

A glance at the HV-list shows:

\begin{proposition}
 Any basic distributive minimal representation-infinite algebra $A$ is defined by zero-relations and by at most three commutativity relations. $A$ is a zero-relation 
algebra if it is obtained by glueing a zero-relation algebra.
\end{proposition}

At the end we discuss shortly the different cases. Of course we can always choose a critical line or a critical algebra $C$ of minimal cardinality. 

Case I is solved completely by Ringel in \cite{Rinmin}. Only special biserial algebras occur and so all glueings are tame and Ringel also studies the module categories.

Case II is more complicated and there is in general no chance to describe the module categories 
as the 
above example shows. Nevertheless this case seems to allow a classification  into finitely 
many families.  I started this project by finding necessary conditions on $R$ ensuring that the algebra is minimal representation-infinite, but I finally 
flinched from producing another list.

Case III means
to classify the minimal representation-infinite ray categories with at most $9$ points and this is  a finite problem.
But already the case of $3$ points treated by Fischbacher in his diploma thesis published in \cite{Fisch3} leads to very many algebras
 and this shows that a general classification makes no sense.

\end{document}